\documentclass[10pt]{article}
\usepackage{amsfonts,amssymb,amsmath, amsthm}

\theoremstyle{plain}
\newtheorem{assumption}{Assumption}
\newtheorem{theorem}{Theorem}[section]
\newtheorem{corollary}[theorem]{Corollary}
\newtheorem{lemma}[theorem]{Lemma}
\newtheorem{proposition}[theorem]{Proposition}

\theoremstyle{definition}
\newtheorem{definition}[theorem]{Definition}
\newtheorem{example}[theorem]{Example}
\newtheorem{examples}[theorem]{Examples}
\newtheorem{preremark}[theorem]{Remark}
\newenvironment{remark}{\begin{preremark}\normalfont}{\end{preremark}}

\allowdisplaybreaks

\newcommand{\e}{\mathcal E}
\newcommand{\F}{\mathcal F}

\newcommand{\R}{\mathbb R}

\begin{document}
\title{Scale-invariant Boundary Harnack Principle in Inner Uniform Domains}
\author{Janna Lierl \and Laurent Saloff-Coste}
\maketitle

\begin{abstract}
We prove a scale-invariant boundary Harnack principle in inner uniform domains in the context of local, regular Dirichlet spaces. For inner uniform Euclidean domains, our results apply to divergence form operators that are not necessarily symmetric, and complement earlier results by H.~Aikawa and A.~Ancona.
\end{abstract}

{\bf Acknowledgement:} Research partially supported by NSF grant DMS 1004771.
{\bf Address:} Malott Hall, Department of Mathematics, Cornell University, Ithaca, NY 14853, United States.

\section*{Introduction}

The boundary Harnack principle is a property of a domain that provides control over the ratio of two harmonic functions in that domain near some part of the boundary where the two functions vanish. Whether a given domain satisfies the boundary Harnack principle depends on the geometry of its boundary and, in fact, there is more than one kind of boundary Harnack principle. For a Euclidean domain $\Omega$, two versions found in the literature are as follows.
\begin{enumerate}
\item 
We say that the \emph{boundary Harnack principle} holds on $\Omega$ if, for any domain $V$  and any compact $K\subset V$ intersecting the boundary $\partial \Omega$, there exists a positive constant $A=A(\Omega,V,K)$ such that for any two positive functions $u$ and $v$ that are harmonic in $\Omega$ and vanish continuously (except perhaps on a polar set) along $V \cap \partial \Omega$, we have
 \[  \frac{u(x)}{u(x')} \leq A \frac{v(x)}{v(x')}, \quad \forall x, x' \in K \cap \Omega.\]
\item 
We say that the \emph{scale-invariant boundary Harnack principle} holds on $\Omega$, if there exist positive constants $A_0$, $A_1$ and $R$, depending only on $\Omega$, with the following property. Let $\xi \in \partial \Omega$ and $r \in (0,R)$. Then
for any two positive functions $u$ and $v$ that are harmonic in $B(\xi,A_0 r) \cap \Omega$ and vanish continuously (except perhaps on a polar set) along $B(\xi,A_0 r) \cap \partial \Omega$, we have
 \[  \frac{u(x)}{u(x')} \leq A_1 \frac{v(x)}{v(x')}, \quad \forall x, x' \in B(\xi,r) \cap \Omega. \]
\end{enumerate}

A third version, important for our purpose and perhaps more natural, would replace the Euclidean balls in (ii) by the inner balls of the domain $\Omega$.

A property similar to (i) was first introduced by Kemper (\cite{Kem72}). The scale-invariant boundary Harnack principle (ii) on Lipschitz domains was proved independently in \cite{Anc78} and \cite{Wu78}, a non-scale invariant version was proved in \cite{Dah77}.

Bass and Burdzy (\cite{BB91}) used probabilistic arguments to prove property (i) on what they call twisted H\"older domains of order $\alpha \in (1/2,1]$. The scale-invariant boundary Harnack principle is in general not true on H\"older domains.
Aikawa (\cite{Aik01}) proved the scale-invariant boundary Harnack principle on uniform domains in Euclidean space. This result was extended to inner uniform domains in \cite{ALM03}. Ancona gave a different proof of these results in \cite{Anc07}. For a summary of the relationships between the inner uniformity conditions and other conditions appearing in the literature, see \cite{Aik04}. Other works on the boundary Harnack principle include \cite{BV96a, BV96b}.

In \cite{GyryaSC}, Gyrya and Saloff-Coste generalized Aikawa's approach to uniform domains in strictly local, symmetric Dirichlet spaces of Harnack-type that admit a carr\'e du champ. Moreover, they deduced that the boundary Harnack principle also holds on inner uniform domains, by considering the inner uniform domain as a uniform domain in a different metric space, namely the completion of the inner uniform domain with respect to its inner metric.

In this paper, we extend the result of \cite{GyryaSC} in two directions. First, we consider "Dirichlet forms" that allow lower order terms and non-symmetry. Second, we prove the boundary Harnack principle directly on inner uniform domains without assuming the existence of a carr\'e du champ. 

We follow Aikawa's approach, but with the Euclidean distance replaced by the inner distance of the domain. A crucial Lemma in our proof concerns the relation between balls in the inner metric and connected components of balls in the metric of the ambient space, see Lemma \ref{lem:metrics are comparable}. This relation was already used in \cite{Anc07} to prove a boundary Harnack principle on inner uniform domains in Euclidean space. Ancona (\cite{Anc07}) also treated second order uniformly elliptic operators with some lower order terms, under the additional condition that the domain is uniformly regular. Following Aikawa's line of reasoning, we do not need the domain to be uniformly regular. 

Our main result is Theorem \ref{thm2:bHP for u}. We now explain how it applies to Euclidean space. Formally, let 
\begin{align} \label{eq:L}
Lf = \sum_{i,j=1}^n \partial_j (a_{i,j} \partial_i f)
           - \sum_{i=1}^n b_i \partial_i f  + \sum_{i=1}^n \partial_i (d_i f) - c f.
\end{align}
Assume that the coefficients $a=(a_{i,j})$, $b=(b_i)$, $d=(d_i)$, $c$ are smooth and satisfy $c - \mathrm{div} b \geq 0$, $c - \mathrm{div} d \geq 0$, and, $\forall \xi \in \R^n$,
$\sum_{i,j} a_{i,j} \xi_i \xi_j \geq \epsilon |\xi|^2$, $\epsilon >0$. 

\begin{theorem}
Let $L$ be the operator defined above and let $\Omega \subset \R^n$ be an inner uniform domain. There exists $C=C(\Omega) > 0$ and for each $R \in (0, C \cdot \emph{\textrm{diam}}(\Omega))$ there exist $A_0, A_1 \in (0,\infty)$, depending on $\Omega$, $R$ and on the coefficients $a$, $b$, $c$ and $d$, such that the scale-invariant boundary Harnack principle holds in the following form. For any $\xi \in \partial_{\Omega} \Omega$, $r \in (0,R)$ and any two positive functions $u$ and $v$ that are local weak solutions of $Lu=0$ in $B_{\widetilde\Omega}(\xi,A_0 r) \cap \Omega$ and vanish weakly along $B_{\widetilde\Omega}(\xi,A_0 r) \cap \partial_{\Omega} \Omega$, we have
 \[  \frac{u(x)}{u(x')} \leq A_1 \frac{v(x)}{v(x')},\quad \forall x, x' \in B_{\widetilde\Omega}(\xi,r). \]
Moreover, if $b=d=c=0$ then the constants $A_0$ and $A_1$ are independent of $R$.
\end{theorem}

Here, by a \emph{local weak solution} $u$ on a domain $U \subset \R^n$ we mean a function that is locally in the Sobolev space $W^1(U)$ of all functions in $L^2(U)$ whose distributional first derivatives can be represented by functions in $L^2(U)$, and satisfies
$\int Lu \, \psi = 0$ for all test functions $\psi$ in $W^1_0(U)$, the closure of $C^{\infty}_0(U)$ (the space of all smooth, compactly supported functions on $U$) in the $W^1$-norm $\Vert \cdot \Vert^2_2 + \Vert \nabla \cdot \Vert^2_2$. A weak solution $u$ \emph{vanishes weakly} along $\partial \Omega$ if $u$ is locally in $W^1_0(\Omega)$ near $U \cap \partial \Omega$.
See Section \ref{ssec:weak solutions}. The definition of a ball $B_{\widetilde\Omega}$ in the inner metric is given in Section \ref{ssec:uniform domains}, $\partial_{\Omega}B_{\widetilde\Omega}$ denotes the boundary of the ball with respect to its completion in the inner metric.

In Section \ref{sec:preliminaries} and \ref{sec:non-sym forms}, we review some general properties of Dirichlet spaces and describe the conditions that we impose on the space. Moreover, we state a localized version of the parabolic Harnack inequality for local weak solutions of the heat equation for second-order differential operators with lower order terms.
In Section \ref{sec:GF} we prove estimates for the heat kernel on balls and for the capacity of balls. After recalling the definition and some properties of inner uniform domains, we give estimates for Green functions on inner balls intersected with an inner uniform domain. In Section \ref{sec:BHP}, we give a proof of the boundary Harnack principle.

\section{Preliminaries} \label{sec:preliminaries}

\subsection{Local weak solutions} \label{ssec:weak solutions}

Let $X$ be a connected, locally compact, separable metrizable space, and let $\mu$ be a non-negative Borel measure on $X$ that is finite on compact sets and positive on non-empty open sets. 
Let $(\e,D(\e))$ be a strictly local, regular, symmetric Dirichlet form on $L^2(X,\mu)$. We denote by $(L,D(L))$ and $(P_t)_{t\geq0}$ the infinitesimal generator and the semigroup, respectively, associated with $(\e,D(\e))$. See \cite{FOT94}.

There exists a measure-valued quadratic form $d\Gamma$ defined by
 \[ \int f \, d\Gamma(u,u) = \e(uf,u) - \frac{1}{2} \e(f,u^2), \quad \forall f,u \in D(\e) \cap L^{\infty}(X), \]
and extended to unbounded functions by setting $\Gamma(u,u) = \lim_{n \to \infty} \Gamma(u_n,u_n)$, where $u_n = \max\{\min\{u,n\},-n\}$. Using polarization, we obtain a bilinear form $d\Gamma$. In particular,
\[ \e(u,v) = \int d\Gamma(u,v), \quad \forall u,v \in D(\e). \]

Let $U \subset X$ be open and connected. Set
 \[ \F_{\mbox{\tiny{{loc}}}}(U)  =  \{ f \in L^2_{\mbox{\tiny{loc}}}(U) : \forall \textrm{ compact } K \subset U, \ \exists f^{\sharp} \in D(\e),
                      f = f^{\sharp}\big|_K \textrm{ a.e.} \} \]
For $f,g \in \F_{\mbox{\tiny{loc}}}(U)$ we define $\Gamma(f,g)$ locally by $\Gamma(f,g)\big|_K = \Gamma(f^{\sharp},g^{\sharp})\big|_K$, where $K \subset U$ is compact and $f^{\sharp},g^{\sharp}$ are functions in $D(\e)$ such that $f = f^{\sharp}$, $g = g^{\sharp}$ a.e.~on $K$. 

Define
\begin{align*}
 \F(U)   &=  \{ u \in \F_{\mbox{\tiny{loc}}}(U) : \int_U |u|^2 d\mu + \int_U d\Gamma(u,u) < \infty \}, \\
\F_{\mbox{\tiny{c}}}(U)  &=  \{ u \in \F(U) : \textrm{ The essential support of } u \textrm{ is compact in } U \}.
\end{align*}

The \emph{intrinsic distance} $d := d_{\e}$ induced by $(\e,D(\e))$ is defined as
\[ d_{\e}(x,y) := \sup \big\{ f(x)-f(y): f \in \F_{\mbox{\tiny{loc}}}(X) \cap C(X), \, d\Gamma(f,f) \leq d\mu \big\},  \]
for all $x,y \in X$, where $C(X)$ is the space of continuous functions on $X$.
Consider the following properties of the intrinsic distance that may or may not be satisfied. They are discussed in \cite{Stu95geometry, SturmI}.
\begin{align}
& \textrm{The intrinsic distance } d  \textrm{ is finite everywhere, continuous, and defines } \nonumber\\
& \textrm{the original topology of } X. \tag{A1} \\
& (X,d) \textrm{ is a complete metric space. } \tag{A2}
\end{align}
Note that if (A1) holds true, then (A2) is by \cite[Theorem 2]{Stu95geometry} equivalent to 
\begin{align}
\forall x \in X, r > 0, \textrm{ the open ball } B(x,r) \textrm{ is relatively compact in } (X,d). \tag{A2'}
\end{align}
Moreover, (A1)-(A2) imply that $(X,d)$ is a geodesic space, i.e.~any two points in $X$ can be connected by a minimal geodesic in $X$. See \cite[Theorem 1]{Stu95geometry}.
If (A1) and (A2) hold true, then by \cite[Proposition 1]{SturmI},
\[  d_{\e}(x,y) := \sup \big\{ f(x)-f(y): f \in D(\e) \cap C_{\mbox{\tiny{c}}}(X), \, d\Gamma(f,f) \leq d\mu \big\}, \quad x,y \in X. \]
It is sometimes sufficient to consider property (A2') only on a subset $Y$ of $X$, that is,
\begin{align} \label{eq:completeness property}
\textrm{For any ball } B(x,2r) \subset Y, \ B(x,r) \textrm{ is relatively compact.} \tag{A2-$Y$}
\end{align}

\begin{definition} Let $V$ be an open subset of $U$. Set
\begin{align*}
 \F^0_{\mbox{\tiny{loc}}}(U,V) = & \{ f \in L^2_{\mbox{\tiny{loc}}}(V,\mu): \forall \textrm{ open } A \subset V \textrm{ rel.~compact in } \overline{U} \textrm{ with } \\
& \, d_{U}(A,U \setminus V) > 0, \, \exists f^{\sharp} \in \F^0(U): f^{\sharp} = f \, \mu \textrm{-a.e.~on } A \},
\end{align*}
where 
 \[ d_U(A, U \setminus V) = \inf \big\{ \textrm{length}(\gamma) \big| \gamma:[0,1] \to U \textrm{ continuous}, \gamma(0) \in A, \gamma(1) \in U \setminus V \big\} \]
 and
\[ \textrm{length}(\gamma) = \sup \left\{ \sum_{i=1}^n d_{\e}(\gamma(t_i),\gamma(t_{i-1})) : n \in \mathbb{N}, 0 \leq t_0 < \ldots < t_n \leq 1 \right\}. \]
\end{definition}

\begin{remark}
Suppose (A1), (A2-$Y$) are satisfied. Let $U \subset Y$. Then $d_U = d_{\e^D_U}$, where $\e^D_U$ is the Dirichlet-type form on $U$ defined in Definition \ref{def1:e^D_U} below. See, e.g., \cite{GyryaSC}.
\end{remark}

\begin{definition}
Let $V \subset U$ be open and $f \in \F_{\mbox{\tiny{c}}}(V)'$, the dual space of $ \F_{\mbox{\tiny{c}}}(V)$ (identify $L^2(X,\mu)$ with its dual space using the scalar product). A function $u: V \to \R$ is a \emph{local weak solution} of the Laplace equation $-Lu = f$ in $V$, if 
\begin{enumerate}
\item
$u \in \F_{\mbox{\tiny{loc}}}(V)$,
\item 
For any function $\phi \in \F_{\mbox{\tiny{c}}}(V), \ \e(u,\phi) = \int f \phi \, d\mu$.
\end{enumerate} 
If in addition
 \[ u \in \F^0_{\mbox{\tiny{loc}}}(U,V), \]
then $u$ is a local weak solution with \emph{Dirichlet boundary condition} along $\partial U$.
\end{definition}

For a time interval $I$ and a Hilbert space $H$, let $L^2(I \to H)$ be the Hilbert space of those functions $v: I \to H$ such that 
 \[  \Vert v \Vert_{L^2(I \to H)} = \left( \int_I \Vert v(t) \Vert_H^2 \, dt \right)^{1/2} < \infty. \]
Let $W^1(I \to H) \subset L^2(I \to H)$ be the Hilbert space of those functions $v: I \to H$ in $L^2(I \to H)$ whose distributional time derivative $v'$ can be represented by functions in $L^2(I \to H)$, equipped with the norm
 \[  \Vert v \Vert_{W^1(I \to H)} = \left( \int_I \Vert v(t) \Vert_H^2 + \Vert v'(t) \Vert_H^2 \, dt \right)^{1/2} < \infty. \]
Let
 \[ \F(I \times X) = L^2(I \to D(\e)) \cap W^1(I \to D(\e)'), \]
where $D(\e)'$ denotes the dual space of $D(\e)$.
Let
 \[ \F_{\mbox{\tiny{loc}}}(I \times U) \] 
be the set of all functions $u:I \times U \to \R$ such that for any open interval $J$ that is relatively compact in $I$, and any open subset $A$ relatively compact in $U$, there exists a function $u^{\sharp} \in \F(I \times U)$ such that $u^{\sharp} = u$ a.e. in $J \times A$.

Let
\[ \F_{\mbox{\tiny{c}}}(I \times U) = \{ u \in \F_{\mbox{\tiny{loc}}}(I \times U): u(t, \cdot) \textrm{ has compact support in } U \textrm{ for a.e. } t \in I \}. \]
For an open subset $V \subset U$, let $Q = I \times V$ and let 
 \[ \F^0_{\mbox{\tiny{loc}}}(U,Q) \]
be the set of all functions $u:Q \to \R$ such that for any open interval $J$ that is relatively compact in $I$, and any open set $A \subset V$ relatively compact in $\overline{U}$ with $d_U(A, U \setminus V) > 0$, there exists a function $u^{\sharp} \in \F^0(I \times U)$ such that $u^{\sharp} = u$ a.e. in $J \times A$.

\begin{definition}
Let $I$ be an open interval and $V \subset U$ open. Set $Q = I \times V$. A function $u: Q \to \R$ is a \emph{local weak solution} of the heat equation $\frac{\partial}{\partial t}u = Lu$ in $Q$, if
\begin{enumerate}
\item
$u \in \F_{\mbox{\tiny{loc}}}(Q)$,
\item 
For any open interval $J$ relatively compact in $I$,
 \[ \forall \phi \in \F_{\mbox{\tiny{c}}}(Q),\  \int_J \int_V \frac{\partial}{\partial t} u \, \phi  \, d\mu \, dt + \int_J \e(u(t,\cdot),\phi(t,\cdot)) dt = 0. \]
\end{enumerate} 
If in addition
 \[ u \in \F^0_{\mbox{\tiny{loc}}}(U,Q), \]
then $u$ is a local weak solution with \emph{Dirichlet boundary condition} along $\partial U$.
\end{definition}

\subsection{Volume doubling, Poincar\'e inequality, and Harnack inequality}

We say that $(X,\mu)$ satisfies the \emph{volume doubling property} on $Y$, if there exists a constant $D_Y \in (0,\infty)$ such that for every ball $B(x,2r) \subset Y$,
\begin{align} \label{eq:VD}
V(x,2r) \leq D_Y \, V(x,r), \tag{VD}
\end{align}
where $V(x,r) = \mu( B(x,r))$ denotes the volume of $B(x,r)$.

The symmetric Dirichlet space $(X,\mu,\e,D(\e))$ satisfies the weak \emph{Poincar\'e inequality} on $Y$ if there exists a constant $P_Y \in (0,\infty)$ such that for any ball $B(x,2r) \subset Y$,
\begin{align} \label{eq:PI}
\forall f \in D(\e), \  \int_{B(x,r)} |f - f_B|^2 d\mu  \leq  P_Y \, r^2 \int_{B(x,2r)} d\Gamma(f,f), \tag{PI}
\end{align}
where $f_B = \frac{1}{V(x,r)} \int_{B(x,r)} f d\mu$ is the mean of $f$ over $B(x,r)$.

For any $s \in \R$, $\tau > 0$, $\delta \in (0,1)$ and $B(x,2r) \subset Y$, define
\begin{align*}
 I &= \big(s - \tau r^2, s \big) \\
 B &= B(x,r) \\
 Q &= I \times B \\
 Q_-  &= \big( s - (3+\delta)\tau r^2/4, s - (3-\delta)\tau r^2/4 \big) \times \delta B \\
 Q_+  &= \big( s - (1+\delta)\tau r^2/4, s \big) \times \delta B.
\end{align*}

\begin{definition} \label{def:HI}
The operator $(L,D(L))$ satisfies the \emph{parabolic Harnack inequality} on $Y$ if, for any $\tau > 0$, $\delta \in (0,1)$, there exists a constant $H_Y(\tau, \delta) \in (0,\infty)$ such that for any ball $B(x,2r) \subset Y$, any $s \in \R$, and any positive function $u \in \F_{\mbox{\tiny{loc}}}(Q)$ with $\frac{\partial}{\partial t} = Lu$ weakly in $Q$, the following inequality holds.
\begin{align} \label{eq:parabolic HI}
\sup_{z \in Q_-} u(z)  \leq  H_Y \inf_{z \in Q_+} u(z) \tag{PHI}
\end{align}
Here both the supremum and the infimum are essential, i.e.~computed up to sets of measure zero.
\end{definition}
The parabolic Harnack inequality implies the \emph{elliptic Harnack inequality},
\begin{align}
\sup_{z \in B(x,r)} u(z) \leq H'_Y \inf_{z \in B(x,r)} u(z), \tag{EHI}
\end{align}
where $u$ is any positive function in $\F_{\mbox{\tiny{loc}}}(Q)$ with $Lu = 0$ weakly in $B(x,2r)$.
Recall also that (PHI) implies the H\"older continuity of local weak solutions.

\begin{theorem} \label{thm:main thm}
Let $(X, \mu, \e, D(\e))$ be a strictly local, regular, symmetric Dirichlet space. Assume that the intrinsic distance $d_{\e}$ satisfies \emph{(A1)-(A2)}.
Then the following properties are equivalent:
\begin{itemize}
\item
$(L,D(L))$ satisfies the parabolic Harnack inequality on $X$.
\item
The volume doubling condition and the Poincar\'e inequality are satisfied on $X$.
\item
The semigroup $(P_t)_{t>0}$ admits an integral kernel $p(t,x,y)$, $t>0$, $x,y \in X$, and there exist constants $c_1, c_2, c_3, c_4 > 0$ such that
 \[ \frac{ c_1 }{ V(x,\sqrt{t}) } \exp \left(- \frac{ d_{\e}(x,y)^2 }{ c_2 t } \right) \leq  p(t,x,y)
\leq  \frac{ c_3 }{ V(x,\sqrt{t}) } \exp \left(- \frac{ d_{\e}(x,y)^2 }{ c_4 t } \right) \]
for all $x,y \in X$ and all $t>0$.
\end{itemize}
\end{theorem}

\begin{proof}
For a detailed discussion see \cite{SturmI}, \cite{SturmII}, \cite{SturmIII}, and \cite{SC02}.
\end{proof}

The following theorem is a special case of Theorem \ref{thm2:local VD+PI = local HI} below.
\begin{theorem} \label{thm:sym local VD + local PI = local HI}
Let $(X,\mu,\e,D(\e))$ be a strictly local, regular, symmetric Dirichlet space and $Y \subset X$. Suppose that $(\e,D(\e))$ satisfies \emph{(A1)}, \emph{(A2-$Y$)}, the volume doubling property \emph{(VD)} on $Y$ and the Poincar\'e inequality \emph{(PI)} on $Y$. Then $(L,D(L))$ satisfies the parabolic Harnack inequality on $Y$. The Harnack constant depends only on $D_Y, P_Y, \tau, \delta$.
\end{theorem}

\begin{definition}
If each point $x \in X$ has a neighborhood $Y_x$ for which the hypotheses of the above theorem are satisfied, then we say that the space is \emph{locally of Harnack-type}.
\end{definition}

\begin{examples}
\begin{enumerate}
\item
Let $(M,g)$ be a Riemannian manifold of dimension $n$. Since $M$ is locally Euclidean, it is locally of Harnack-type. Suppose the Ricci curvature of $M$ is bounded below, that is, there is a constant $K \geq 0$ so that the Ricci tensor is bounded below by $\mathcal{R} \geq - K g$. Then the volume doubling condition and the Poincar\'e inequality hold uniformly on all balls $Y_x = B(x,r)$, $x \in M$, $r \in (0,R)$, with constants $D_M$ and $P_M$ depending on $\sqrt{K}R$, hence the parabolic Harnack inequality holds. See \cite[Section 5.6.3]{SC02}. In particular, if $K=0$ then volume doubling and Poincar\'e inequality hold true globally with scale-invariant constants.
\item
Let $M$ be a complete locally compact length-metric space of finite Hausdorff dimension $n\geq 2$. $M$ is called an \emph{Alexandrov space}, if its curvature is bounded below by some $K \in \R$ in the following sense. For any two points $x,y \in M$, let $\gamma_{xy}$ be a minimal geodesic joining $x$ to $y$ with parameter proportional to the arc-length. Then for any triangle $\bigtriangleup xyz$ consisting of the three geodesics $\gamma_{xy}$, $\gamma_{yz}$, $\gamma_{zx}$, there exists a comparison triangle $\bigtriangleup \tilde x \tilde y \tilde z$ in a simply connected space of constant curvature $K$ such that
 \[ d(x,y) = d(\tilde x, \tilde y), \ d(y,z) = d(\tilde y, \tilde z), \ d(z,x) = d(\tilde z, \tilde x)  \]
and
 \[ d\big(\gamma_{xy}(s),\gamma_{xz}(t)\big) \geq d\big(\gamma_{\tilde x \tilde y}(s),\gamma_{\tilde x \tilde z}(t)\big) \quad \textrm{ for any } s,t \in [0,1]. \]
Alexandrov spaces arise naturally as limits (in the Gromov-Hausdorff topology) of sequences of closed Riemannian manifolds $M(n,K,D)$ of dimension $n$, diameter at most $D$, and with sectional curvature bounded below by $K \in \R$. 

On any Alexandrov space there is a canonical strictly local, regular, symmetric Dirichlet form $(\e,D(\e))$ on $L^2(M,\mathcal{H}^n)$, where $\mathcal{H}^n$ is the Hausdorff measure in dimension $n$, given by
\begin{align*}
 \e(f,g) & = \int_M \langle \nabla f, \nabla g \rangle \, d\mathcal{H}^n, \\
 D(\e) & = W^1_0(M).
\end{align*}
The inner product $\langle \cdot, \cdot \rangle$, the gradient $\nabla$ and the Sobolev space $W^1_0(M)$ are Riemannian like objects that are provided by the Alexandrov space structure. Concrete descriptions of these objects as well as of the associated infinitesimal generator (Laplacian) are given in \cite{KMS01}.

Let $Y \subset M$ be open and relatively compact. Like in the case of a manifold with Ricci curvature bounded below, it is proved in \cite{KMS01} that the Dirichlet form $(\e,D(\e))$ satisfies the volume doubling condition and the Poincar\'e inequality on $Y$, as well as conditions (A1) and (A2-$Y$).
\item
Let $\Omega$ be an open, connected subset of $\R^n$. 
Let $X_i$, $0\leq i \leq k$, be smooth vector fields on $R^n$ which satisfy H\"ormanders condition, that is, there is an integer $N$ such that at any point $x$ in $\Omega$, the vectors $X_i(x)$ and all their brackets of order less than $N+1$ span the tangent space at $x$. Let $\omega$ be a smooth positive function on $\R^n$ such that $\omega$ and $\omega^{-1}$ are bounded. Then the symmetric Dirichlet form
\begin{align*}
 \e(f,g) & =  \int_{\Omega} \sum_{i=1}^k X_i f \, X_i g \, \omega \, d\mu, \quad f,g \in D(\e),
\end{align*}
where the domain $D(\e)$ is the closure of $C^{\infty}_0(\Omega)$ in the $(\e(\cdot,\cdot) + \Vert \cdot \Vert_2)$-norm,
is sub-elliptic. That is, for any relatively compact set $U$ there exist constants $c$, $\epsilon$ such that 
\[ \e(f,f) \geq c \Vert f \Vert_{2,\epsilon}^2, \quad f \in C^{\infty}_0(\Omega), \]
where $\Vert f \Vert_{2,\epsilon}^2 = \int |\hat f(\xi)|^2 (1+|\xi|^2)^{\epsilon} d\xi$. See \cite{Hoe67}.

The distance $\rho_{\e}$ induced by $(\e,D(\e))$ satisfies conditions (A1)-(A2), see \cite{JS87}. Moreover, the Poincar\'e inequality, \cite{Jer86}, and the volume doubling condition, \cite{NSW85}, hold true locally. 
\end{enumerate}
\end{examples}

\section{The Dirichlet form $( \e,D( \e))$} \label{sec:non-sym forms}

\subsection{Non-symmetric forms}

\begin{definition}
Let $(\e,D(\e))$ be a bilinear form on $L^2(X,\mu)$. Let $\e^{\mbox{\tiny{sym}}}(u,v) = \frac{1}{2}(\e(u,v) + \e(v,u))$ be its symmetric part and $\e^{\mbox{\tiny{skew}}}(u,v) = \frac{1}{2}(\e(u,v) - \e(v,u))$ its skew-symmetric part. Then $(\e,D(\e))$ is called a \emph{coercive closed form}, if 
\begin{enumerate}
\item
$D(\e)$ is a dense linear subspace of $L^2(X,\mu)$,
\item
$(\e^{\mbox{\tiny{sym}}},D(\e))$ is a positive definite, closed form on $L^2(X,\mu)$,
\item
$(\e,D(\e))$ satisfies the \emph{sector condition}, i.e.~there exists a constant $C_0 > 0$ such that
\begin{align*}
 |\e^{\mbox{\tiny{skew}}}(u,v)| & \leq C_0 \big( \e(u,u) + (u,u) \big)^{1/2} \big( \e(v,v) + (v,v) \big)^{1/2},
\end{align*}
for all $u,v \in D(\e)$. 
\end{enumerate}
\end{definition}
Coercive closed forms are discussed in \cite{MR92}. Every coercive closed form $(\e,D(\e))$ is associated uniquely with an infinitesimal generator $(L, D(L))$ and a strongly continuous contraction semigroup $(P_t)_{t>0}$. Furthermore, the form
\begin{align*}
& \e^*(f,g) := \e(g,f), \\ 
& D(\e^*) := D(\e).
\end{align*}
is also a coercive closed form. Its infinitesimal generator $(L^*, D(L^*))$ is the adjoint operator of $(L,D(L))$, and for its semigroup $(P^*_t)_{t>0}$, $P^*_t$ is the adjoint of $P_t$ for each $t>0$. If these semigroups admit continuous kernels $p^*$ and $p$, respectively, then the kernels are related by
\[ p^*(t,x,y) = p(t,y,x), \quad \forall t>0, \, \forall x,y \in X. \]

For any $f \in L^2(X,\mu)$, let $f^+ = \max\{f,0\}$ and $f \wedge 1 = \min\{f,1\}$.
\begin{definition}
A \emph{Dirichlet form} $(\e,D(\e))$ is a coercive closed form such that for all $u \in D(\e)$ we have $u^+ \wedge 1 \in D(\e)$ and the following two inequalities hold,
\begin{equation} \label{eq:Markovian}
\begin{split}
 \e(u + u^+ \wedge 1, u - u^+ \wedge 1) \geq 0, \\
 \e(u - u^+ \wedge 1, u + u^+ \wedge 1) \geq 0.
\end{split}
\end{equation}
This definition is equivalent to the property that the semigroup $(P_t)_{t>0}$ associated with the coercive closed form $(\e,D(\e))$ and its adjoint $(P^*_t)_{t>0}$ are both sub-Markovian.
\end{definition}

The symmetric part $\e^{\mbox{\tiny{sym}}}$ of a local, regular Dirichlet form can be written uniquely as 
\[ \e^{\mbox{\tiny{sym}}}(f,g) = \e^{\mbox{\tiny{s}}}(f,g) + \int fg \, d\kappa,  \quad \textrm{ for all } f,g \in D(\e), \]
where $\e^{\mbox{\tiny{s}}}$ is strictly local and $\kappa$ is a positive Radon measure. Let $\Gamma$ be the energy measure of the strictly local part $\e^{\mbox{\tiny{s}}}$.

\begin{example}
On Euclidean space, consider the form
\begin{align*}
 \e(f,g) & = \int \! \sum_{i,j=1}^n  a_{i,j} \partial_i f \partial_j g \, dx
           + \int \! \sum_{i=1}^n b_i \partial_i f \, g \, dx
           + \int \! \sum_{i=1}^n f \, d_i \partial_i g \, dx
           + \int \! c f g \, dx,
\end{align*}
where the coefficients $a=(a_{i,j})$, $b=(b_i)$, $d=(d_i)$, $c$ are bounded and measurable with $c - \mbox{div }b \geq 0$ and $c - \mbox{div }d \geq 0$ in the distribution sense, and, $\forall \xi \in \R^n$, $\sum_{i,j} a_{i,j} \xi_i \xi_j \geq \epsilon |\xi|^2$, $\epsilon >0$.  Then $(\e,D(\e))$ with domain $D(\e)=W^1_0(\R^n)$ is a Dirichlet form.

Set $\tilde a_{i,j} := (a_{i,j} + a_{j,i})/2$ and $\check a_{i,j} = (a_{i,j} - a_{j,i})/2$. Then the symmetric part of $\e$ is 
\begin{align*}
 \e^{\mbox{\tiny{sym}}}(f,g) 
 = & \int \sum_{i,j=1}^n \tilde a_{i,j} \partial_i f \partial_j g \, dx
           + \int \sum_{i=1}^n \frac{b_i+d_i}{2} \partial_i f \, g \, dx \\
   &        + \int \sum_{i=1}^n f \, \frac{b_i+d_i}{2} \partial_i g \, dx
       + \int c f g \, dx,
\end{align*}
while the skew-symmetric part of $\e$ is 
\begin{align*}
 \e^{\mbox{\tiny{skew}}}(f,g) 
 = & \int \sum_{i,j=1}^n \check a_{i,j} \partial_i f \partial_j g \, dx
           + \int \sum_{i=1}^n \frac{b_i-d_i}{2} \partial_i f \, g \, dx \\
   &        + \int \sum_{i=1}^n f \, \frac{-b_i + d_i}{2} \partial_i g \, dx.
\end{align*}
The symmetric part $\e^{\mbox{\tiny{sym}}}$ can be decomposed into its strictly local part
\begin{align*}
 \e^{\mbox{\tiny{s}}}(f,g) & = \sum_{i,j=1}^n \int \tilde a_{i,j} \partial_i f \partial_j g \, dx
\end{align*}
and its killing part, where $\kappa$ is given by
\[ \int \psi \,d\kappa =  \frac{1}{2} \int (c - \mathrm{div } b + c - \mathrm{div } d) \psi \, dx, \quad \psi \in C^{\infty}_0(\R^n). \] 
\end{example}

\subsection{Assumptions on the forms}

Let $X$ be a locally compact, separable metrizable space and let $\mu$ be a non-negative Borel measure on $X$ that is finite on compact sets and positive on non-empty open sets.
We fix a symmetric, strictly local, regular Dirichlet form $(\hat\e,D(\hat\e))$ on $L^2(X,\mu)$ with energy measure $\hat\Gamma$.   
Let $Y \subset X$ and assume that the intrinsic metric $d=d_{\hat\e}$ satisfies (A1)-(A2-$Y$). 

Let $(\e,D(\e))$ be a (possibly non-symmetric) local bilinear form on $L^2(X,\mu)$.

\begin{assumption} \label{as2:e_t} 
\begin{enumerate}
\item
$(\e,D(\e))$ is a local, regular Dirichlet form with domain $D(\e) = D(\hat \e)$. Let $C_0$ be the constant in the sector condition for $(\e,D(\e))$.
\item
There is a constant $C_1 \in (0,\infty)$ so that for all $t>0$ and all $f,g \in \F_{\mbox{\emph{\tiny{loc}}}}(Y)$ with $fg \in \F_{\mbox{\emph{\tiny{c}}}}(Y)$,
\begin{align*}
C_1^{-1} \int f^2 d\hat\Gamma(g,g) \leq \int f^2 d\Gamma(g,g) \leq C_1 \int f^2 d\hat\Gamma(g,g),
\end{align*}
where $\Gamma$ is the energy measure of $\e^{\mbox{\emph{\tiny{s}}}}$.
\item
There are constants $C_2, C_3 \in [0,\infty)$ so that for all $f \in \F_{\mbox{\emph{\tiny{loc}}}}(Y)$ with $f^2 \in \F_{\mbox{\emph{\tiny{c}}}}(Y)$,
\begin{align*}
\int f^2 d\kappa 
\leq 2 \left(\int f^2 d\mu \right)^{\frac{1}{2}} \left(C_2 \int d\hat\Gamma(f,f) + C_3 \int f^2 d\mu \right)^{\frac{1}{2}}
\end{align*}
\item
There are constants $C_4, C_5 \in [0,\infty)$ such that for all $f \in \F_{\mbox{\emph{\tiny{loc}}}}(Y)$, $g \in \F_{\mbox{\emph{\tiny{c}}}}(Y) \cap L^{\infty}(Y)$,
\begin{align*}
 \left| \e^{\mbox{\emph{\tiny{skew}}}}(f,fg^2) \right|
\leq & 2 \left( \int f^2 d\hat\Gamma(g,g) \right)^{\frac{1}{2}} \left( C_4 \int g^2 d\hat\Gamma(f,f) + C_5 \int f^2 g^2 d\mu \right)^{\frac{1}{2}}.
\end{align*}
\end{enumerate}
\end{assumption}

\begin{assumption} \label{as2:p=0}
There are constants $C_6, C_7 \in [0,\infty)$ such that
\begin{align*} \big|\e^{\mbox{\emph{\tiny{skew}}}}(f,f^{-1} g^2)\big| 
 \leq & \, 2 \left( \int d\hat\Gamma(g,g) \right)^{\frac{1}{2}} \left( C_6 \int g^2 d\hat\Gamma(\log f,\log f)  \right)^{\frac{1}{2}} \\
 & + 2 \left( \int d\hat\Gamma(g,g) + \int g^2 d\hat\Gamma(\log f,\log f) \right)^{\frac{1}{2}} \left( C_7  \int g^2 d\mu \right)^{\frac{1}{2}},
\end{align*}
for all $f \in \F_{\mbox{\emph{\tiny{loc}}}}(Y)$ with $f + f^{-1} \in L^{\infty}_{\mbox{\emph{\tiny{loc}}}}$, and all $g \in \F_{\mbox{\emph{\tiny{c}}}}(Y) \cap L^{\infty}(Y)$.
\end{assumption}

\begin{remark}
\begin{enumerate}
\item
Assumptions \ref{as2:e_t} and \ref{as2:p=0} are more restrictive than Assumptions 2 and 3 in \cite{LierlSC2}. Here, we assume in addition that $(\e,D(\e))$ is a time-independent Dirichlet form. In particular, $(\e,D(\e))$ is positive definite and Markovian. 
\item
Assumption \ref{as2:e_t}\emph{(ii)} holds if and only if for all $f \in \F_{\mbox{\tiny{c}}}(Y)$, 
\[ C_1^{-1} \hat\e(f,f) \leq \e^{\mbox{\tiny{s}}}(f,f) \leq C_1 \hat\e(f,f). \]
See, e.g., \cite{Mosco94}.
\item
$\e$ satisfies the above assumptions if and only if the adjoint $\e^*(f,g):=\e(g,f)$ satisfies them.
\item
If Assumption \ref{as2:e_t}(iv) is satisfied with $C_4=0$, then Assumption \ref{as2:p=0} is satisfied with $C_6=0$. To see this, apply Assumption \ref{as2:e_t}(iv) to $\e^{\mbox{\tiny{skew}}}_t(f,f^{-1}g^2) = \e^{\mbox{\tiny{skew}}}_t(f,f (f^{-1}g)^2)$.
\item
Assumptions \ref{as2:e_t} and \ref{as2:p=0} are satisfied by the classical forms on Euclidean space associated with example given in the introduction. The constants $C_4$, $C_6$ can be taken to be equal to $0$ only if $a_{i,j}$ is symmetric for all $i,j$, and $C_2$, $C_5$, $C_7$ can be taken to be equal to $0$ only if $b_i=d_i=0$ for all $i$ (i.e., if there is no drift term).
\end{enumerate}
\end{remark}

Let 
 \[ C_8 := C_2+C_3+C_5+C_7. \]
Let $(L,D(L))$ be the infinitesimal generator of $(\e,D(\e)$.

\subsection{Parabolic Harnack inequality}
Let $(X,\mu,\hat\e,D(\hat\e))$ be a strictly local, regular Dirichlet space and $Y \subset X$. Let $(\e,D(\e))$ satisfy Assumptions \ref{as2:e_t} and \ref{as2:p=0}. Let $(L,D(L))$ be the infinitesimal generator associated with $(\e,D(\e))$.

\begin{definition} Let $V \subset U \subset X$ be open subsets. Let $f \in \F_{\mbox{\tiny{c}}}(V)$. A function $u: V \to \R$ is a \emph{local weak solution} of $Lu = f$ in $V$, if 
\begin{enumerate}
\item
$u \in \F_{\mbox{\tiny{loc}}}(V)$
\item 
For any function $\phi \in \F_{\mbox{\tiny{c}}}(V), \ \e(u,\phi) = \int f \phi d\mu$.
\end{enumerate} 
If in addition
 \[ u \in \F^0_{\mbox{\tiny{loc}}}(U,V), \]
then $u$ is a local weak solution with \emph{Dirichlet boundary condition} along $\partial U$.
\end{definition}

\begin{definition}
Let $I$ be an open interval and $V \subset U$ open. Set $Q = I \times V$. A function $u: Q \to \R$ is a \emph{local weak solution} of the heat equation $\frac{\partial}{\partial t} = Lu$ in $Q$, if
\begin{enumerate}
\item
$u \in \F_{\mbox{\tiny{loc}}}(Q)$,
\item 
For any open interval $J$ relatively compact in $I$,
 \[ \forall \phi \in \F_{\mbox{\tiny{c}}}(Q),\  \int_J \int_V \frac{\partial}{\partial t} u \, \phi  \, d\mu \, dt + \int_J \e(u(t,\cdot),\phi(t,\cdot)) dt = 0. \]
\end{enumerate} 
If in addition
 \[ u \in \F^0_{\mbox{\tiny{loc}}}(U,Q), \]
then $u$ is a local weak solution with \emph{Dirichlet boundary condition} along $\partial U$.
\end{definition}
Analogously to Definition \ref{def:HI}, we can describe the elliptic and parabolic Harnack inequalities for local weak solutions of $Lu=0$ and $\frac{\partial}{\partial t} = Lu$.

\begin{remark}
An equivalent definition of a local weak solution of $\frac{\partial}{\partial t} u = L u$ on $Q = I \times V$ is
\begin{enumerate}
\item
$u \in L^2(I \to D(\e))$,
\item
For any open interval $J$ relatively compact in $I$,
\begin{align*}
 -\int_J \int_V \frac{\partial}{\partial t} \phi \, u \, d\mu \, dt + \int_J \e(u(t,\cdot),\phi(t,\cdot)) dt = 0,
\end{align*}
for all $\phi \in \F(Q)$ with compact support in $J \times V$.
\end{enumerate}
See \cite{ESC}.
\end{remark}

\begin{theorem}  \label{thm2:local VD+PI = local HI}
Let $(X,\mu,\hat\e,D(\hat\e))$ and $(\e,D(\e))$ be as above and $Y \subset X$. Suppose that $(\e,D(\e))$ satisfies Assumptions \ref{as2:e_t}, \ref{as2:p=0}, and $(\hat\e,D(\hat\e))$ satisfies \emph{(A1)}, \emph{(A2-$Y$)}, the volume doubling property \emph{(VD)} on $Y$ and the Poincar\'e inequality \emph{(PI)} on $Y$. Then $L$ satisfies the parabolic Harnack inequality on $Y$. The Harnack constant depends only on $D_Y$, $P_Y$, $\tau$, $\delta$, $C_1$-$C_7$ and an upper bound on $C_8 r^2$.
\end{theorem}

\begin{proof}
See \cite{LierlSC2}.
\end{proof}

\begin{corollary}
Let $(X,\mu,\hat\e,D(\hat\e))$, $(\e,D(\e))$ and $Y \subset X$ be as in Theorem \ref{thm2:local VD+PI = local HI}.
Fix $\tau > 0$ and $\delta \in (0,1)$. Then there exist $\beta \in (0,1)$ and $H \in (0,\infty)$ such that for any $B(x,2r) \subset Y$, $s > 0$, any local weak solution of $\frac{\partial}{\partial t}u = Lu$ in $Q = (s - \tau r^2,s) \times B(x,r)$ has a continuous representative and satisfies
 \[ \sup_{(t,y),(t',y')\in Q_-} \left\{ \frac{ |u(t,y) - u(t',y')| }{ [ |t-t'|^{1/2} + d_{\e}(y,y')^{\beta} ] } \right \}
\leq \frac{H}{r^{\beta} } \sup_{Q} |u| \]
where $ Q_-  = ( s - (3+\delta)\tau r^2/4, s - (3-\delta)\tau r^2/4 ) \times B(x,\delta r)$.
The constant $H$ depends only on $D_Y$, $P_Y$, $\tau$, $\delta$, $C_1$-$C_7$ and an upper bound on $C_8 r^2$.
\end{corollary}

\begin{proof}
See, e.g., \cite{SC02}.
\end{proof}

\section{Green functions estimates and inner uniformity} \label{sec:GF}

Let $(X, \mu, \hat\e, D(\hat\e))$ be a symmetric, strictly local, regular Dirichlet space and $Y \subset X$. Suppose (A1)-(A2-$Y$), the volume doubling condition (VD) on $Y$ and the Poincar\'e inequality (PI) on $Y$ hold.
Suppose that $(\e,D(\e))$ satisfies Assumptions \ref{as2:e_t} and \ref{as2:p=0}.
Recall that by Theorem \ref{thm2:local VD+PI = local HI}, $L$ and $L^*$ satisfy (PHI) on $Y$.

\subsection{Dirichlet-type Dirichlet form and heat kernel}

\begin{definition} \label{def1:e^D_U}
Let $U$ be an open subset of $X$. The Dirichlet-type form on $U$ is defined as
\[ \e^D_U(f,g) = \e(f,g),  \quad f,g \in D(\e^D_U), \]
where the domain $D(\e^D_U) = \F^0(U)$ is the closure of the space $C^{\infty}_{\mbox{\tiny{c}}}(U)$ of all smooth functions with compact support in $U$. The closure is taken in the norm $\e^D_{U,1}(f,f)^{\frac{1}{2}} = \left(\e^D_U(f,f) + \int_U f^2 d\mu \right)^{\frac{1}{2}}$.
\end{definition}

The form $(\e^D_U,D(\e^D_U))$ is associated with a semigroup $P^D_U(t)$, $t > 0$.
Using the reasoning in \cite[Section 2.4]{SturmII}, one can show that the semigroup has a continuous kernel $p^D_U(t,x,y)$. Moreover, the map $y \mapsto p^D_U(t,x,y)$ is in $\F^0(U)$.

\subsection{Capacity}

For $\alpha > 0$, let
 \[ \e_{\alpha}( \cdot, \cdot) = \e(\cdot,\cdot) + \alpha (\cdot,\cdot)_{L^2}. \]
For any open set $A$ in an open, relatively compact set $U \subset X$  define
 \[ \mathcal{L}_{A,U} = \{ w \in D( \e^D_U) : w \geq 1 \textrm{ a.e. on } A \}. \] 
If $\mathcal{L}_{A,U} \neq \emptyset$, there exist unique functions $e_{A,\alpha}, \hat e_{A,\alpha} \in \mathcal{L}_{A,U}$ such that for all $w \in \mathcal{L}_{A,U}$ it holds
\begin{equation} \label{eq:e_A and hat e_A}
 \e_{\alpha}(e_{A,\alpha},w) \geq  \e_{\alpha}(e_{A,\alpha},e_{A,\alpha}) \textrm{  and  }  \e_{\alpha}(w, \hat e_{A,\alpha}) \geq  \e_{\alpha}(\hat e_{A,\alpha}, \hat e_{A,\alpha}).
\end{equation}
Notice that this implies that $\e_{\alpha}(e_{A,\alpha}, \hat e_{A,\alpha}) = \e_{\alpha}(e_{A,\alpha},e_{A,\alpha}) = \e_{\alpha}(\hat e_{A,\alpha}, \hat e_{A,\alpha})$.
Moreover, for any open $A$ such that $\mathcal{L}_{A,U} \neq \emptyset$, $e_{A,\alpha}$ is the smallest function $u$ on $U$ such that $u \wedge 1$ is a $\alpha$-excessive function in $D( \e^D_U)$ and $u \geq 1$ on $A$. See \cite[Proposition III.1.5]{MR92}.

The \emph{$\alpha$-capacity} (with respect to $(\e, D(\e))$) of $A$ in $U$ is defined by
 \[ {\textrm{Cap}}_{U,\alpha} (A) = \begin{cases} \e_{\alpha}(e_{A,\alpha},e_{A,\alpha}), & \mathcal{L}_{A,U} \neq \emptyset \\
                                                   + \infty,      & \mathcal{L}_{A,U} = \emptyset. 
                                     \end{cases}
\]
The $\alpha$-capacity is extended to non-open sets $A \subset U$ by
 \[ {\textrm{Cap}}_{U,\alpha} (A) = \inf \{ {\textrm{Cap}}_{U,\alpha} (B): A \subset B \subset U, \, B \textrm{ open} \}. \]

The \emph{$0$-capacity} is defined similarly, with $\e_{\alpha}$ replaced by $\e$ and $D( \e^D_U)$ replaced by the extended Dirichlet space $\F_e$. $\F_e$ is defined (see \cite{FOT94}) as the family of all measurable, almost everywhere finite functions $u$ such that there exists an approximating sequence $u_n \in D( \e^D_U)$ that is $\e^D_U$-Cauchy and $u = \lim u_n$ almost everywhere. 

By \cite[Proposition VI.4.3]{BG68}, $e_{A,0} =  G^D_U \nu_A$, where $\nu_A$ is a finite measure with $\textrm{supp}(\nu_A)$ contained in the completion of $A$, and $ G^D_U$ is the Green function associated with $ \e^D_U$ (see \cite[page 256]{BG68}). Thus,
 \[ {\textrm{Cap}}_{U,0} (A)
 =  \e(e_{A,0},e_{A,0}) 
 =  \e( G^D_U \nu_A,e_{A,0})
  = \int e_{A,0} \, d\nu_A
 = \nu_A(U).
 \]

Let $\widetilde{\mbox{Cap}}_{U,\alpha} (A) = \e^{\mbox{\tiny{s}}}_{\alpha}(e^{\mbox{\tiny{s}}}_{A,\alpha}, e^{\mbox{\tiny{s}}}_{A,\alpha})$ be the $\alpha$-capacity with respect to the strictly local part $\e^{\mbox{\tiny{s}}}$ of the symmetric part $\e^{\mbox{\tiny{sym}}}$.

\begin{lemma} \label{lem:Cap and tilde Cap}
Let $\alpha \in (0,1]$.
For any set $A$ in $U \subset Y$, 
 \[ \widetilde{\emph{\textrm{Cap}}}_{U,1} (A) \leq {\emph{\textrm{Cap}}}_{U,1} (A) \leq C^2 \, \widetilde{\emph{\textrm{Cap}}}_{U,1} (A), \]
where $C = (1 + C_0 C_1/\alpha) \left( 1+\inf_{\epsilon} \{ \epsilon/\alpha + \frac{1}{\epsilon} \max \{ C_2, C_3/\alpha \}  \} \right)$.
\end{lemma}

\begin{proof}
It suffices to consider an open set $A \subset U$.
By \eqref{eq:e_A and hat e_A}, the Cauchy-Schwarz inequality, the sector condition and Assumption \ref{as2:e_t},
\begin{align*}
 \e_{\alpha}(e_{A,\alpha},e_{A,\alpha})
& \leq   \e_{\alpha}(e_{A,\alpha},e^{\mbox{\tiny{s}}}_{A,\alpha}) \\
& \leq  (1 + C_0 C_1/\alpha) \Big( \e_{\alpha}(e^{\mbox{\tiny{s}}}_{A,\alpha},e^{\mbox{\tiny{s}}}_{A,\alpha}) \Big)^{1/2} \Big(  \e_{\alpha}(e_{A,\alpha},e_{A,\alpha}) \Big)^{1/2} \\
 &  \leq  C \Big( \e^{\mbox{\tiny{s}}}_{\alpha}(e^{\mbox{\tiny{s}}}_{A,\alpha},e^{\mbox{\tiny{s}}}_{A,\alpha}) \Big)^{1/2} \Big(  \e_{\alpha}(e_{A,\alpha},e_{A,\alpha}) \Big)^{1/2},
\end{align*}
where $C = (1 + C_0 C_1/\alpha) \left( 1+\inf_{\epsilon} \{ \epsilon/\alpha + \frac{1}{\epsilon} \max \{ C_2, C_3/\alpha \}  \} \right)$.
Hence,
 \[  {\textrm{Cap}}_{U,\alpha} (A) =  \e_{\alpha}(e_{A,\alpha},e_{A,\alpha}) 
\leq C^2 \e^{\mbox{\tiny{s}}}_{\alpha}(e^{\mbox{\tiny{s}}}_{A,\alpha},e^{\mbox{\tiny{s}}}_{A,\alpha}) = C^2 \widetilde{\textrm{Cap}}_{U,\alpha} (A).  \]
On the other hand, by \eqref{eq:e_A and hat e_A} and the Cauchy-Schwarz inequality,
\begin{align*}
\e^{\mbox{\tiny{s}}}_{\alpha}(e^{\mbox{\tiny{s}}}_{A,\alpha},e^{\mbox{\tiny{s}}}_{A,\alpha})
& \leq  \e^{\mbox{\tiny{s}}}_{\alpha}(e_{A,\alpha},e^{\mbox{\tiny{s}}}_{A,\alpha})
 \leq \Big( \e^{\mbox{\tiny{s}}}_{\alpha}(e^{\mbox{\tiny{s}}}_{A,\alpha},e^{\mbox{\tiny{s}}}_{A,\alpha}) \Big)^{1/2} \Big(  \e^{\mbox{\tiny{s}}}_{\alpha}(e_{A,\alpha},e_{A,\alpha}) \Big)^{1/2} \\
& \leq \Big( \e^{\mbox{\tiny{s}}}_{\alpha}(e^{\mbox{\tiny{s}}}_{A,\alpha},e^{\mbox{\tiny{s}}}_{A,\alpha}) \Big)^{1/2} \Big(  \e_{\alpha}(e_{A,\alpha},e_{A,\alpha}) \Big)^{1/2}.
\end{align*}
Therefore,
 \[  \widetilde{\textrm{Cap}}_{U,\alpha} (A) = \e^{\mbox{\tiny{s}}}_{\alpha}(e^{\mbox{\tiny{s}}}_{A,\alpha},e^{\mbox{\tiny{s}}}_{A,\alpha}) 
\leq \e_{\alpha}(e_{A,\alpha},e_{A,\alpha}) = \textrm{Cap}_{U,\alpha} (A).  \]
\end{proof}

\begin{theorem}
\label{thm:capacity estimate}
Suppose $(X,\mu,\hat\e,D(\hat\e))$ satisfies \emph{(A1)-(A2-$Y$)}, \emph{(VD)} on $Y$ and \emph{(PI)} on $Y$, and $(\e,D(\e))$ satisfies Assumptions \ref{as2:e_t} and \ref{as2:p=0}.
Then there are constants $a,A \in (0,\infty)$ such that for any $r \in (0,R)$ and any ball $B(x,2R) \subset Y$ we have
\begin{equation} \label{eq:capacity estimate}
A^{-1} \int_r^R \frac{s}{V(x,s)} ds \leq \left( {\mbox{\em{Cap}}}_{B(x,R),0} \big( B(x,r) \big) \right)^{-1} \leq A \int_r^R  \frac{s}{V(x,s)} ds.
\end{equation}
The constant $A$ depends only on $D_Y$, $P_Y$, and an upper bound on 
\[ (1 + C_0 C_1 / \alpha)^4 \left( 1+\inf_{\epsilon} \{ \epsilon / \alpha + \frac{1}{\epsilon} \max \{ C_2, C_3/\alpha \}  \} \right)^2, \]
where $\alpha = \min\{1,\lambda_R\}$ and $\lambda_R$ is the smallest Dirichlet eigenvalue of $-L^{\mbox{\emph{\tiny{sym}}}}$ on $B(x,R)$.
\end{theorem}

\begin{proof}
Let $r \in (0,R)$ and $B = B(x,r)$. 
First, consider the estimate
\begin{equation}
A^{-1} \int_r^R \frac{s}{V(x,s)} ds \leq \left( \widetilde{\mbox{Cap}}_{B(x,R),0} \big( B(x,r) \big) \right)^{-1} \leq A \int_r^R  \frac{s}{V(x,s)} ds.
\end{equation}
The lower bound is proved in \cite[Theorem 1]{Stu95geometry} using the strict locality of $\e^{\mbox{\tiny{s}}}$. The upper bound can be proved as in \cite[Lemma 4.3]{GSC02} using the heat kernel estimates of Theorem \ref{thm3:basic p^D_B estimate} below. 

If $(\e,D(\e))$ is symmetric and strictly local, then $\textrm{Cap}_{B(x,R),0} (B)$ is the same as $\widetilde{\textrm{Cap}}_{B(x,R),0} (B)$, hence the assertion follows. Otherwise, we show that the two $0$-capacities are comparable.

In view of Lemma \ref{lem:Cap and tilde Cap}, it suffices to show that
 \[  \textrm{Cap}_{B(x,R),0} (B) 
 \asymp  \textrm{Cap}_{B(x,R),\alpha} (B).  \] 
 and 
 \[  \widetilde{\textrm{Cap}}_{B(x,R),0} (B) 
 \asymp  \widetilde{\textrm{Cap}}_{B(x,R),\alpha} (B),  \]
for some $\alpha \in (0,1]$.  
Let
\[ \lambda_R := \inf_{0 \neq f \in \F^0(B(x,R))} \frac{ \e^D_{B(x,R)}(f,f) }{ \int f^2  d\mu } > 0 \]
be the lowest Dirichlet eigenvalue of $-L^{\mbox{\tiny{sym}}}$ on ${B(x,R)}$, and $\alpha = \min \{ 1, \lambda_R \}$. Then for any $f \in \F^0(B(x,R))$,
 \[ \e^D_{B(x,R)}(f,f) \leq \e^D_{B(x,R),\alpha}(f,f) \leq 2 \e^D_{B(x,R)}(f,f). \]
Let $f \in  \F_e^{\e^D_{B(x,R)}}$. Then there is an approximating sequence $(f_n)$ in $\F^0(B(x,R))$ such that $\e^D_{B(x,R)}(f_n - f_m, f_n - f_m) \to 0$ as $n,m  \to \infty$, and $f_n \to f$ almost everywhere. Thus,
 \[ \e^D_{B(x,R),\alpha} (f_n - f_m, f_n - f_m) \leq 2 \e^D_{B(x,R)}(f_n - f_m, f_n - f_m) \to 0. \]
So $ \F_e^{\e^D_{B(x,R)}} =  \F^0(B(x,R))$. Hence, by \eqref{eq:e_A and hat e_A} and the sector condition,
\begin{align*}
{\textrm{Cap}}_{B(x,R),\alpha}(B)  
 = \e_{\alpha} (e_{B,\alpha},  e_{B,\alpha}) 
& \leq (1+C_0 C_1/\alpha)^2 \e_{\alpha} (e_{B,0}, e_{B,0}) \\
& \leq 2 (1+C_0 C_1/\alpha)^2 {\textrm{Cap}}_{B(x,R),0}(B).
\end{align*}
On the other hand,
\begin{align*}
 \e (e_{B,0}, e_{B,0}) 
 & \leq \e (e_{B,0}, e_{B,\alpha}) \leq (1+C_0 C_1/\alpha) \, \e_{\alpha} (e_{B,0}, e_{B,0})^{1/2} \e_{\alpha} (e_{B,\alpha}, e_{B,\alpha})^{1/2}  \\
 & \leq \sqrt{2} (1+C_0 C_1/\alpha) \, \e(e_{B,0}, e_{B,0})^{1/2} \e_{\alpha} (e_{B,\alpha}, e_{B,\alpha})^{1/2}.
\end{align*}
Hence,
 \[ {\textrm{Cap}}_{B(x,R),0}(B) \leq 2 (1+C_0 C_1/\alpha)^2 {\textrm{Cap}}_{B(x,R),\alpha}(B). \] 
Similar, we can show that
 \[  \widetilde{\textrm{Cap}}_{B(x,R),0} (B) 
 \asymp  \widetilde{\textrm{Cap}}_{B(x,R),\alpha} (B).  \]
\end{proof}

\begin{remark}
The Dirichlet eigenvalue $\lambda_R$ is bounded below by 
\[ \lambda_R \geq \frac{C}{R^2} \]
for some constant $C>0$ depending on $D_Y$ and $P_Y$. See \cite[Theorem 2.6]{HSC01}.
\end{remark}

From now on, we only consider the $0$-capacity, and thus drop the index $0$.

\subsection{(Inner) uniformity}  \label{ssec:uniform domains}

Let $\Omega \subset X$ be open and connected.
The \emph{inner metric} on $\Omega$ is defined as
 \[ d_{\Omega}(x,y) = \inf \big\{ \textrm{length}(\gamma) \big| \gamma:[0,1] \to \Omega \textrm{ continuous}, \gamma(0) = x, \gamma(1) = y \big\}. \]
Let $\widetilde\Omega$ be the completion of $\Omega$ with respect to $d_{\Omega}$. 
Whenever we consider an inner ball $B_{\widetilde\Omega}(x,R) = \{ y \in \widetilde\Omega : d_{\Omega}(x,y)<R \}$ or $B_{\Omega}(x,R) = B_{\widetilde\Omega}(x,R) \cap \Omega$, we assume that its radius is minimal in the sense that $B_{\widetilde\Omega}(x,R) \neq B_{\widetilde\Omega}(x,r)$ for all $r < R$. Let $\partial_{\Omega} B_{\widetilde\Omega}(x,r)$ be the boundary of the ball with respect to its completion in the inner metric.
If $x$ is a point in $\Omega$, denote by $\delta_{\Omega}(x) = d(x,X \setminus \Omega)$ the distance from $x$ to the boundary of $\Omega$.

\begin{definition}
\begin{enumerate}
\item 
Let $\gamma: [\alpha,\beta] \to \Omega$ be a rectifiable curve in $\Omega$ and let $c \in (0,1)$, $C \in (1,\infty)$. We call $\gamma$ a $(c,C)$-\emph{uniform} curve in $\Omega$ if
\begin{equation}
\delta_{\Omega} \big( \gamma(t) \big) \geq c \cdot \min \left\{ d \big( \gamma(\alpha), \gamma(t) \big) , d \big( \gamma(t), \gamma(\beta) \big) \right\},    \quad \textrm{ for all  } t \in [\alpha,\beta],
\end{equation}
and if
\[ \textrm{length}(\gamma) \leq C \cdot d \big( \gamma(\alpha), \gamma(\beta) \big). \]
The domain $\Omega$ is called $(c,C)$-\emph{uniform} if any two points in $\Omega$ can be joined by a $(c,C)$-uniform curve in $\Omega$.
\item
\emph{Inner uniformity} is defined analogously by replacing the metric $d$ on $X$ with the inner metric $d_{\Omega}$ on $\Omega$. 
\item
The notion of \emph{(inner) $(c,C)$-length-uniformity} is defined analogously by replacing $d(\gamma(a),\gamma(b))$ by $\textrm{length} (\gamma\big|_{[a,b]})$.
\end{enumerate}
\end{definition}

The next proposition is taken from \cite[Proposition 3.3]{GyryaSC}. See also \cite[Lemma 2.7]{MS79}.
\begin{proposition} 
Assume that $(X,d)$ is a complete, locally compact length metric space with the property that there exists a constant $D$ such that for any $r > 0$, the maximal number of disjoint balls of radius $r/4$ contained in any ball of radius $r$ is bounded above by $D$. Then any connected open subset $U \subset X$ is uniform if and only if it is length-uniform.
\end{proposition}

Let $\Omega$ be a $(c_u,C_u)$-inner uniform domain in $(X,d)$. 
\begin{lemma} \label{lem:x_r} 
For every ball $B = B_{\widetilde\Omega}(x,r)$ in $(\widetilde\Omega,d_{\Omega})$ with minimal radius, there exists a point $x_r \in B$ with $d_{\Omega}(x,x_r) = r/4$ and $d(x_r,X \setminus \Omega) \geq c_ur/8$.
\end{lemma}

\begin{proof} This is immediate, see \cite[Lemma 3.20]{GyryaSC}.
\end{proof}

The following lemma is crucial for the proof of the boundary Harnack principle on inner uniform domains, rather than uniform domains. A version of this lemma was already used in \cite{Anc07} to prove a boundary Harnack principle on inner uniform domains in Euclidean space.

Let $p:\widetilde\Omega \to \overline{\Omega}$ be the natural projection. For any $x \in \widetilde {\Omega}$ and any ball $D=B(p(x),r)$, let $D'$ be the connected component of $\widetilde{\Omega}$ that contains $x$ and so that $D' \cap \Omega$ is a connected component of $D \cap \Omega$.

\begin{lemma} \label{lem:metrics are comparable}
Suppose $\mu$ has the volume doubling property on $Y \subset X$.
Then there exists a positive constant $C_{\Omega}$ such that for any ball $D = B(p(x),r)$ with $x \in \widetilde\Omega$ and $B(p(x),4r) \subset Y$,
 \[ B_{\widetilde\Omega}(x,r) \subset D' \subset B_{\widetilde\Omega}(x,C_{\Omega} r). \]
The constant $C_{\Omega}$ depends only on $D_Y$ and the inner uniformity constants $c_u$, $C_u$ of $\Omega$.
\end{lemma}

\begin{remark}
For any $x \in \Omega$, $r>0$,
 \[ D' \cap \Omega = \{ y \in \Omega : d_{\textrm{diam}}(x,y) \leq r \}, \]
where the \emph{inner diameter metric} $d_{\textrm{diam}}$ is defined as
 \[ d_{\textrm{diam}}(x,y) := \inf\{ \textrm{diam}(\gamma) : \gamma \textrm{ path from } x \textrm{ to } y \textrm{ in }\Omega \}, \]
and the diameter is taken in the metric $d$ of the underlying space $(X,d)$. On Euclidean space, Lemma \ref{lem:metrics are comparable} is immediate from the fact that the inner diameter metric is equivalent to the inner (length) metric $d_{\Omega}$, see \cite[Theorem 3.4]{Vae98}.
\end{remark}

\begin{proof}[Proof of Lemma \ref{lem:metrics are comparable}.]
Clearly, $B_{\Omega}(x,r) \subset D'$. To show the second inclusion, we follow the line of reasoning given in \cite[Proof of Theorem 3.4]{Vae98}. Replacing $r$ by a slightly larger radius, we may assume that $x \in \Omega$. Let $y \in D' \cap \Omega$ and let $\alpha$ be a path in $D' \cap \Omega$ connecting $x$ to $y$. There exist finitely many points $x=x_1, x_2,\ldots,x_N=y$ on the path $\alpha$ so that $d_{\Omega}(x_{j-1},x_j) = d(x_{j-1},x_j)$ for all $2 \leq j \leq N$. 
Let $M \leq 2r$ be the diameter of $\alpha$ in $(X,d)$. By Lemma \ref{lem:x_r} each $x_j$ can be joined to a point $y_j = (x_j)_{2M} \in \Omega$ with $d(y_j,\widetilde\Omega \setminus \Omega) \geq c_u M/4$ by a path $\alpha_j$ of length $d_{\Omega}(y_j,x_j) \leq M/2$. Set
$Y_0 = \{ y_j : 1 \leq j \leq N \}$ and
 \[ U = \bigcup_j B(y_j,c_u M/4). \]
Let $\mathcal{P}$ be the family of connected components of $U$. There exists a constant $C = C(D_Y,c_u)$ such that for each $j$, we have
\[ V(y_j,c_u M/4) \geq C \, V(y_j,3M/2) \geq C \, V(x,M). \]
 Hence $\sharp \mathcal{P} \cdot C \, V(x,M) \leq \mu(U) \leq V(x,2M)$ and 
\begin{align*}
\sharp \mathcal{P} \leq C'(D_Y,c_u).
\end{align*}

We claim that if $y,y' \in Y_0$ are in the same component $V$ of $U$, then there exists a path $\beta$ connecting $y$ to $y'$ in $V$ such that $\textrm{length}(\beta) \leq c_1 M$ for some constant $c_1 > 0$ depending only on $c_u$ and $D_Y$.
For $z \in Y$ we write $B(z) = B(z,c_u M / 4)$. Since $V$ is connected, there is a finite sequence $y = z_0, \ldots , z_k = y' \in Y_0 \cap V$ such that $B(z_{i-1}) \cap B(z_i) \neq \emptyset$ for all $1 \leq i \leq k$. Passing to a subsequence we may assume that the balls $B(z_i)$ with even $i$ are disjoint. Since there are at least $k/2$ of these balls, we get
\[ \frac{k}{2} C \, V(x,M) \leq \mu(V) \leq  \mu(U) \leq V(x,2M), \]
so $k \leq C''(D_Y,c_u)$. For each $i$, we can connect $z_{i-1}$ to $z_i$ by a path $\beta_i$ in $\Omega$ of length at most $c_u M/2$. Now the conjunction of the paths $\beta_i$ is a path $\beta$ of length at most
\begin{align}
\textrm{length}(\beta) \leq k c_u M/2 \leq c_1 M. 
\end{align}

We define integers $0 = j_0 < j_1 < \ldots < j_s = N$ and distinct components $V_1, \ldots, V_s$ of $U$ as follows. Let $V_1$ be the component that contains $y_1$. Assuming that $j_{n-1}$ and $V_{n-1}$ are defined, we iteratively define $j_n$ to be the largest number $j$ such that $y_j \in V_{n-1}$, and let $V_n$ be the component that contains $y_{j_n+1}$.

For each $1 \leq i \leq s$ we have shown above that there exists a path $\beta_{j_i}$ connecting $y_{j_{i-1}+1}$ to $y_{j_i}$. Let $\gamma$ be the conjunction of these paths, the geodesic segments $[x_{j_i}, x_{j_i + 1}]$, $1 \leq i \leq s - 1$, and the paths $\alpha_m$ for $m = 1, j_1, j_1 + 1, j_2, j_2 + 1,\ldots, j_s = N$. Then $\gamma$ is path in $\widetilde \Omega$ that connects $x$ to $y$ and has length
\[ \textrm{length}(\gamma) \leq s c_1 M + s M + s M/2 \leq C' (c_1+2) M . \]
This means that $D' \subset B_{\widetilde\Omega}(x,C_{\Omega} r)$ with $C_{\Omega} = 2 C' (c_1+2)$.
\end{proof}

\subsection{Green function estimates} \label{ssec:GF's estimates}

Recall that for an open set $U \subset X$, $G_U=G^D_U$ is the Green function and $p^D_U$ is the heat kernel associated with $\big(\e^D_U,D(\e^D_U)\big)$.

\begin{theorem} \label{thm3:basic p^D_B estimate}
Suppose $(X,\mu,\hat\e,D(\hat\e))$ satisfies (A1)-(A2-$Y$), (VD) on $Y$ and (PI) on $Y$, and $(\e,D(\e))$ satisfies Assumptions \ref{as2:e_t} and \ref{as2:p=0}.
Let $B=B(a,R)$ with $B(a,2R) \subset Y$.
\begin{enumerate}
\item
For any fixed $\epsilon \in (0,1)$ there are constants $c, C \in (0,\infty)$ such that for any $x,y \in B(a,(1-\epsilon)R)$ and $0 < \epsilon t \leq R^2$, the Dirichlet heat kernel $p^D_B$ is bounded below by
 \[  p^D_B(t,x,y) \geq \frac{c}{V(x,\sqrt{t} \wedge R_x)} \exp\left( - C \frac{d(x,y)^2}{t} \right), \]
where $R_x = d(x,\bar B \setminus B)/2$.
\item
For any fixed $\epsilon \in (0,1)$ there are constants $c, C \in (0,\infty)$ such that for any $x,y \in B$, $t \geq (\epsilon R)^2$, the Dirichlet heat kernel $p^D_B$ is bounded above by
 \[  p^D_B(t,x,y) \leq \frac{ C }{ V(a,R) } \exp\left(-\frac{c t}{R^2} \right). \]
\item 
There exist constants $c, C \in (0,\infty)$ such that for any $x,y \in B$, $t > 0$, the Dirichlet heat kernel $p^D_B$ is bounded above by
\begin{equation}
 p^D_B(t,x,y) 
\leq C \frac{\exp \left( - c \frac{d(x,y)^2}{t} \right)}
            {V(x,\sqrt{t} \wedge R)^{1/2} V(y,\sqrt{t} \wedge R)^{1/2}}.
\end{equation}
\end{enumerate}
All the constants $c,C$ above depend only on $D_Y$, $P_Y$, $C_1$-$C_7$ and an upper bound on $C_8 R^2$.
\end{theorem}

\begin{proof}
See \cite{LierlSC2}.
\end{proof}

\begin{lemma} \label{lem:Green function exists}
Let $B(a,2R) \subset Y$. Then for any relatively compact, open set $V \subset B(a,R)$, the Green function $y \mapsto G_V(x,y)$ is in $\F^0_{\mbox{\em{\tiny{loc}}}}(V,V\setminus\{x\})$ for any fixed $x \in V$.
\end{lemma}

\begin{proof}
We follow \cite[Lemma 4.7]{GyryaSC}. Recall that the map $y \mapsto p^D_V(t,x,\cdot)$ is in $\F^0(V)$. The heat kernel upper bounds of Theorem \ref{thm3:basic p^D_B estimate} imply that $\psi G_V(x,\cdot) \in L^2(X,\mu)$ for any continuous function $\psi$ with compact support $K$ in $X \setminus \{ x \}$. Indeed, by the set monotonicity of the kernel and Theorem \ref{thm3:basic p^D_B estimate}, there are constants $c, C \in (0,\infty)$, depending on $R$, such that for all $t \geq R^2$ and $z,y \in V$,
\begin{align} \label{eq:4.3}
p^D_V(t,z,y) \leq C e^{-c t / R^2},
\end{align}
and there are constants $c',C' \in (0,\infty)$ depending on $R$ such that for all $t>0$ and $z,y \in V$,
\begin{align} \label{eq:4.4}
p^D_V(t,z,y) \leq C' e^{-c'/t}. 
\end{align}
This shows that the integral $\psi G_V(x,\cdot) = \int_0^{\infty} \psi p^D_V(t,x,\cdot)dt$ converges at $0$ and $\infty$ in $L^2(X,\mu)$. Hence $\psi G_V(x,\cdot)$ is in $L^2(X,\mu)$.

Next, we show that the integral also converges in $\F^0(V)$. Let $\psi$ be as above with the additional property that $d\Gamma(\psi,\psi) \leq d\mu$ on $X$.
For fixed $0 < a < b < \infty$, set $g = \int_a^b p^D_V(t,x,\cdot) dt$ and observe that $\psi g, \psi^2 g \in \F^0(V)$. 
By the Cauchy-Schwarz inequality,
\begin{align*}
\e(\psi g, \psi g)
 \leq & \, 2 \int_{V} g^2 d\Gamma(\psi,\psi) 
       + 2 \int_{V} \psi^2  d\Gamma(g,g) 
       +  \int_{V} \psi^2 g^2 d\kappa \\    
\leq & \, 2 \sup \psi^2 \left( \int_{K \cap V} d\Gamma(g,g) + \int_{K \cap V} g^2 d\kappa \right) \\
&  + 2 \sup \frac{d\Gamma(\psi,\psi)}{d\mu} \int_{K \cap V} g^2 d\mu  \\
\leq & \, C \int (-Lg) g \, d\mu 
    +  2 \int_{K \cap V} g^2 d\mu \\
 = & \, C \int_{K \cap V} g \big( p^D_V(a,x,\cdot) - p^D_V(b,x,\cdot) \big) d\mu 
    +  2 \int_{K \cap V} g^2 d\mu \\
\leq & \, C \int_{K \cap V} g \, p^D_V(a,x,\cdot) d\mu 
    + 2 \int_{K \cap V} g^2 d\mu.
\end{align*}
for some constant $C > 0$ depending on $\sup \psi^2$.
Now, observe that \eqref{eq:4.3}-\eqref{eq:4.4} imply that 
 \[ \int_{K \cap V} g^2 d\mu = \int_{K \cap V} \bigg(\int_a^b p^D_V(t,x,\cdot) dt \bigg)^2 d\mu \]
tends to $0$ when $a$, $b$ tend to infinity or when $a$, $b$ tend to $0$ (this is indeed the argument we used above to show that $G_V(x,\cdot)$ is in $L^2(X,d\mu)$). The same estimates \eqref{eq:4.3}-\eqref{eq:4.4} imply that $\int_{K \cap V} g p^D_V(a,x,\cdot) d\mu$ tends to $0$ when $a$, $b$ tend to infinity or when $a$, $b$ tend to $0$. 
This implies that the integral $\psi G_V(x,y) = \psi \int_0^{\infty} p^D_V(t,x,\cdot)dt$ converges in $\F^0(V)$ as desired.
\end{proof}

\begin{lemma} \label{lem:4.8}
\begin{enumerate}
\item
There is a constant $C$ depending only on $D_Y$, $P_Y$, $C_1$-$C_7$ and an upper bound on $C_8 R^2$, such that for any ball $B(z,2R) \subset Y$,
\begin{equation} \label{eq:GF upper estimate in a ball}
 \forall x,y \in B(z,R), \quad  G_{B(z,R)}(x,y) \leq C \int_{d(x,y)^2/2}^{2R^2} \frac{ ds }{ V(x,\sqrt{s}) }.
\end{equation}
\item
Fix $\theta \in (0,1)$. There is a constant $C$ depending only on $\theta$, $D_Y$, $P_Y$, $C_1$-$C_7$ and an upper bound on $C_8 R^2$, such that for any ball $B(z,2R) \subset Y$, 
\begin{equation} \label{eq:GF lower estimate in a ball}
 \forall x,y \in B(z,\theta R), \quad  G_{B(z,R)}(x,y) \geq C \int_{d(x,y)^2/2}^{2R^2} \frac{ds}{ V(x,\sqrt{s}) }.
\end{equation}
\end{enumerate}
\end{lemma}

\begin{proof}
See \cite[Lemma 4.8]{GyryaSC} and use the estimates of Theorem \ref{thm3:basic p^D_B estimate}.
\end{proof}

Recall that for an open set $U \subset X$, $B_U(x,r) = \{ y \in U : d_U(x,y) < r \}$, where $d_U$ is the inner metric of the domain $U$. Let $G_{B_U(x,r)}$ be the Green function on $B_U(x,r)$.

\begin{lemma} \label{lem:4.9}
Fix $\theta \in (0,1)$. Let $U \subset X$ be an open set.
\begin{enumerate}
\item
There is a constant $C$ depending only on $\theta$, $D_Y$, $P_Y$, $C_1$-$C_7$ and an upper bound on $C_8 R^2$ such that for any $B(z,2R) \subset Y$,
\begin{equation} \label{eq:GF upper estimate in U cap B}
  G_{B_U(z,R)}(x,y) \leq G_{U \cap B(z,R)}(x,y) \leq C \frac{ R^2 }{ V(x,R) },
\end{equation}
for all $x,y \in U \cap B(z,R)$ with $d(x,y) \geq \theta R$. 
\item
Let $U$ be an open subset so that $\overline{U} \subset Y$. Consider a ball $B_U(z,2R) \subset Y$ and suppose that any two points in $B_U(z,\delta R)$ can be connected by a $(c_u,C_u)$-inner uniform curve in $U$, for some $\delta < 1/3$. Then there is a constant $C$ depending only on $\theta$, $D_Y$, $P_Y$, $c_u$, $C_u$, $C_1$-$C_7$ and an upper bound on $C_8 R^2$,  such that 
\begin{equation} \label{eq:GF lower estimate in U cap B}
  G_{B_U(z,R)}(x,y) \geq C \frac{ R^2 }{ V(x,R) },
\end{equation}
for all $x,y \in B_U(z, \delta R)$ with $d(x, X \setminus U)$, $d(y, X \setminus U) \in (\theta R, \infty)$ and $d_U(x,y) \leq \delta R / C_u$.
\end{enumerate}
\end{lemma}

\begin{proof}
We follow the line of reasoning of \cite[Lemma 4.9]{GyryaSC}.
Set $B = B(z,R)$, $W = U \cap B(z,R)$. The upper bound \eqref{eq:GF upper estimate in U cap B} follows easily
from Lemma \ref{lem:4.8} and the monotonicity inequality $G_W \leq G_B$. By assumption, there is an $\epsilon_1 > 0$ such that for any $x,y$ as in (ii), there is a path in $U$ from $x$ to $y$ of length less than
$C_u d_U(x,y) \leq \delta R$ that stays at distance at least $\epsilon_1 R$ from $X \setminus U$. Since $x,y \in B_U(z,\delta R)$ and $\delta < 1/3$, this path is contained in
 \[ B_U(z,R) \cap \{ \zeta \in U : d(\zeta,X \setminus U) > \epsilon_1 R \}. \]
Using this path, the Harnack inequality easily reduces the lower bound \eqref{eq:GF lower estimate in U cap B} to
the case when $y$ satisfies $d(x,y) = \eta R$ for some arbitrary fixed $\eta \in (0,\epsilon_1)$ small enough.
Pick $\eta > 0$ so that, under the conditions of the lemma, the ball $B(x,2\eta R)$ is
contained in $B_U(z,R)$. Let $W=B_U(z,R)$. Then the monotonicity property of Green functions implies that $G_W(x,y) \geq G_{B(x,\eta R)}(x,y)$. Lemma \ref{lem:4.8} and the volume doubling property then yield
 \[  G_W(x,y) \geq C \frac{R^2}{V(x,R)}. \]
This is the desired lower bound.
\end{proof}

\section{Boundary Harnack principle} \label{sec:BHP}

\subsection{Reduction to Green functions estimates}
\label{ssec:bHP}

Let $(X, \mu, \hat\e, D(\hat\e))$ be a symmetric, strictly local, regular Dirichlet space and $Y \subset X$. Suppose (A1)-(A2-$Y$), the volume doubling condition (VD) on $Y$ and the Poincar\'e inequality (PI) on $Y$ hold.
Suppose that $(\e,D(\e))$ satisfies Assumptions \ref{as2:e_t} and \ref{as2:p=0}.
We obtain that under these assumptions, local weak solutions of $Lu=0$ (resp. $L^*u=0$) in $Y$ are harmonic functions for the associated Markov process and satisfy the maximum principle. This can be proved following the line of reasoning given in \cite[Theorem 4.3.2, Lemma 4.3.2]{FOT94} and using \cite[Proposition V.1.6, Proof of Lemma III.1.4]{MR92}.

Let $\Omega$ be a domain so that $\overline{\Omega} \subset Y$. For $\xi \in \widetilde\Omega \setminus \Omega$, set $B_{\Omega}(\xi,r) := B_{\widetilde\Omega}(\xi,r) \cap \Omega$.
Let $c_u \in (0,1)$ and $C_u \in (1,\infty)$. Let $A_3 = 2(12 + 12C_u)$, $A_0 = A_3 + 7$, $A_7 = 2/c_u + 1$, and $A_8 = 2(A_0 \vee 7 A_7)$.
Let $p:\widetilde\Omega \to \overline{\Omega}$ be the natural projection ($p(x)=x$ for $x \in \Omega$). For $\xi \in \widetilde\Omega \setminus \Omega$, let $R_{\xi}$ be the largest radius so that 
\begin{enumerate}
\item
 $B(p(\xi),A_8 R_{\xi}) \subset Y$,
\item
$B_{\widetilde\Omega}(\xi,A_0 R_{\xi}) \neq \widetilde\Omega$,
\item
$12R_{\xi}/c_u \leq  \textrm{diam}_{\Omega}(\Omega)/2$ if $\Omega$ is a bounded domain.
\item
Any two points in $B_{\widetilde\Omega}(\xi,12R_{\xi}/c_u)$ can be connected by a curve that is $(c_u,C_u)$-inner uniform in $\Omega$. 
\end{enumerate}

For $\xi \in \widetilde\Omega \setminus \Omega$ and $0< R \leq R_{\xi}$, let
\[ \lambda_{R,\xi} := \inf \{ \lambda_{R}(z) : z \in B_{\Omega}(\xi',(4(1+2/c_u)R), \xi' \in \widetilde \Omega \setminus \Omega \cap B_{\widetilde\Omega}(\xi,14R_{\xi}) \}, \]
\[ \lambda_{R}(z) = \inf_{0 \neq f \in \F^0(B)} \frac{\e^D_B(f,f)}{\int f^2 d\mu}, \quad \textrm{ where } B=B_{\Omega}(z,2(1+ 2/c_u)R). \]
Let
\[ A_{R} = (1 + C_0 C_1 / \alpha)^4 \left( 1+\inf_{\epsilon} \{ \epsilon / \alpha + \frac{1}{\epsilon} \max \{ C_2, C_3/\alpha \}  \} \right)^2, \]
where $\alpha = \min\{ 1,\lambda_{R,\xi} \}$.

\begin{theorem} \label{thm:bHP with GF's}
There exists a constant $A_1' \in (1,\infty)$ such that for any $\xi \in \widetilde\Omega\setminus \Omega$ with $R_{\xi} > 0$ and any
 \[ 0 < r < R \leq \inf \{ R_{\xi'} : \xi' \in B_{\widetilde\Omega}(\xi,7R_{\xi}) \setminus \Omega \}, \]
we have
 \[ \frac{  G_{Y'} (x,y) }{   G_{Y'} (x',y) }
     \leq A_1' \frac{  G_{Y'} (x,y') }{  G_{Y'} (x',y') }, \]
for all $x,x' \in B_{\Omega}(\xi,r)$ and $y,y' \in \Omega \cap \partial_{\Omega} B_{\widetilde\Omega}(\xi,6r)$. Here $Y' = B_{\widetilde\Omega}(\xi,A_0r)$. 
The constant $A'_1$ depends only on $D_Y$, $P_Y$, $c_u$, $C_u$, $C_0$-$C_7$, and upper bounds on $C_8 R^2$ and $A_R$.
\end{theorem}
The proof of this theorem is the content of Section \ref{ssec:bHP for L} below. It is based on the estimates for the Green functions in Section \ref{ssec:GF's estimates}.

\begin{theorem} \label{thm2:bHP for u}
There exists a constant $A_1 \in (1,\infty)$ such that for any $\xi \in \widetilde\Omega\setminus \Omega$ with $R_{\xi}>0$ and any 
 \[ 0 < r < R \leq \inf \{ R_{\xi'} : \xi' \in B_{\widetilde\Omega}(\xi,7R_{\xi}) \setminus \Omega \}, \]
and any two non-negative weak solutions $u,v$ of $ L u = 0$ in $Y' = B_{\widetilde\Omega}(\xi,A_0r)$ with weak Dirichlet boundary condition along $B_{\widetilde\Omega}(\xi,6r) \setminus \Omega$, we have
 \[  \frac{u(x)}{u(x')} \leq A_1 \frac{v(x)}{v(x')}, \]
for all $x, x' \in B_{\Omega}(\xi,r)$. The constant $A_1$ depends only on the volume doubling constant $D_Y$, the Poincar\'e constant $P_Y$, the constants $C_0$-$C_7$ which give control over the skew-symmetric part and the killing part of the Dirichlet form, the inner uniformity constants $c_u$, $C_u$, and upper bounds on $C_8 R^2$ and $A_R$.
\end{theorem}

\begin{remark}
\begin{enumerate}
\item
The hypothesis that $R_{\xi}>0$ can be understood as ``local inner uniformity''. Clearly, $R_{\xi}>0$ holds true at every boundary point $\xi$ of an inner uniform domain. Since the statement of Theorem \ref{thm2:bHP for u} is local, it is natural to assume that only points near $\xi$ need to be connected by inner uniform curves.
\item
A consequence of Theorem \ref{thm2:bHP for u} is that the ratio $\frac{u}{v}$ of the two local weak solutions $u$ and $v$ is H\"older continuous.
\item
Applications to estimates of the heat kernel with Dirichlet boundary condition in inner uniform domains are given in the forthcoming paper \cite{LierlSC3}.
\end{enumerate}
\end{remark}

\begin{theorem} \label{thm:Martin boundary}
Let $(X,d,\mu,\hat\e,D(\hat\e))$ be a strictly local regular Dirichlet space that satisfies \emph{(A1)}, \emph{(A2-$Y$)}, \emph{(VD)} and \emph{(PI)} on $Y \subset X$. Suppose $(\e,D(\e))$ satisfies Assumptions \ref{as2:e_t} and \ref{as2:p=0}. Let $\Omega \subset Y$ be an inner uniform domain in $(X,d)$. Let $\Omega^*$ be the compactification of $\Omega$ with respect to the inner metric. Then the Martin compactification relative to $(\e,D(\e))$ of $\Omega$ is homeomorphic to $\Omega^*$ and each boundary point $\xi \in \Omega^* \setminus \Omega$ is minimal.
\end{theorem}

\begin{proof}
The assertion can be proved along the line of \cite[Theorem 1.1]{ALM03} using the boundary Harnack principle of Theorem \ref{thm2:bHP for u}.
\end{proof}

\begin{proof}[Proof of Theorem \ref{thm2:bHP for u}.]
Fix $\xi \in \widetilde\Omega\setminus \Omega$ and $0 < r < R$ as in the theorem.
Following the argument given in \cite[Proof of Theorem 1]{Aik01}, we show that for any weak solution $u$ of $ L u=0$ in $Y'$ with weak Dirichlet boundary condition along $B_{\widetilde\Omega}(\xi,6r) \setminus \Omega$, there exists
a Borel measure $\nu_u$ such that
\begin{align} \label{eq:u is a potential}
 u(x) = \int_{\Omega \cap \partial_{\Omega} B_{\widetilde\Omega}(\xi,6r)}  G_{Y'}(x,y) d\nu_u(y)
\end{align}
for all $x \in \Omega \cap B_{\widetilde\Omega}(\xi,r)$, where $ G_{Y'}$ is the Green function corresponding to the Dirichlet form $( \e^D_{Y'},D( \e^D_{Y'}))$. 
Let $\hat R_u^{\Omega \cap \partial_{\Omega} B_{\widetilde\Omega}(\xi,6r)}$ be the lower regularization of the reduced function
\begin{align*}
 R_u^{\Omega \cap \partial_{\Omega} B_{\widetilde\Omega}(\xi,6r)}(x)  
=  \inf \{ v(x) & : v \textrm{ positive and $L$-superharmonic on } Y', \\
                  & \ \  v \geq u \textrm{ on } \Omega \cap \partial_{\Omega} B_{\widetilde\Omega}(\xi,6r) \} 
\end{align*}
of $u$ on $\Omega \cap \partial_{\Omega} B_{\widetilde\Omega}(\xi,6r)$.
Since $u$ is a positive local weak solution of $Lu=0$ on $Y'$, $u = \hat R_u^{\Omega \cap \partial_{\Omega} B_{\widetilde\Omega}(\xi,6r)}$ quasi-everywhere on $\Omega \cap \partial_{\Omega} B_{\widetilde\Omega}(\xi,6r)$, and $L \hat R_u^{\Omega \cap \partial_{\Omega} B_{\widetilde\Omega}(\xi,6r)} = 0$ on $\Omega \setminus \partial_{\Omega} B_{\widetilde\Omega}(\xi,6r)$. Moreover, $\hat R_u^{\Omega \cap \partial_{\Omega} B_{\widetilde\Omega}(\xi,6r)} = 0$ q.e.~on $B_{\widetilde\Omega}(\xi,6r) \setminus \Omega$ by assumption. Hence $u = \hat R_u^{\Omega \cap \partial_{\Omega} B_{\widetilde\Omega}(\xi,6r)}$ on $\Omega \cap B_{\widetilde\Omega}(\xi,6r)$ by the maximum principle. As in \cite[Proof of Theorem 5.3.5]{AG01}, one can show that there is a measure $\nu_u$ supported on $\Omega \cap \partial_{\Omega} B_{\widetilde\Omega}(\xi,6r)$, so that
\[ \hat R_u^{\Omega \cap \partial_{\Omega} B_{\widetilde\Omega}(\xi,6r)}(x) 
= \int_{\Omega \cap \partial_{\Omega} B_{\widetilde\Omega}(\xi,6r)}  G_{Y'}(x,y) d\nu_u(y),
\quad \forall x \in \Omega \cap B_{\widetilde\Omega}(\xi,r). \]
This proves \eqref{eq:u is a potential}.

By Theorem \ref{thm:bHP with GF's}, there exists a constant $A_1' \in (1,\infty)$ such that for all $x,x' \in B_{\Omega}(\xi,r)$ and all $y,y' \in \Omega \cap \partial_{\Omega} B_{\widetilde\Omega}(\xi,6r)$, we have
 \[ \frac{  G_{Y'}(x,y) }{  G_{Y'}(x',y) } \leq A'_1 \frac{  G_{Y'}(x,y') }{  G_{Y'}(x',y') }. \]
For any (fixed) $y' \in \Omega \cap \partial_{\Omega} B_{\widetilde\Omega}(\xi,6r)$, we find that
\begin{align*}
\frac{1}{A'_1} u(x)  
& \leq  \frac{  G_{Y'}(x,y') }{  G_{Y'}(x',y') }   \int_{\Omega \cap \partial_{\Omega} B_{\widetilde\Omega}(\xi,6r)}   G_{Y'}(x',y) d\nu_u(y) \\
& =   \frac{  G_{Y'}(x,y') }{  G_{Y'}(x',y') }  u(x')
   \leq  A'_1 u(x).
\end{align*}
We get a similar inequality for $v$. Thus, for all $x, x' \in B_{\Omega}(\xi,r)$,
\begin{equation} \label{eq:bHP with u and G}
\frac{1}{A'_1} \frac{ u(x) }{ u(x') } 
    \leq \frac{  G_{Y'}(x,y') }{ G_{Y'}(x',y') }   \leq   A'_1 \frac{ v(x) }{ v(x') }.
\end{equation}
\end{proof}

\subsection{Proof of Theorem \ref{thm:bHP with GF's}} \label{ssec:bHP for L}

We follow closely \cite{Aik01} and \cite{GyryaSC}.
Notice that the estimates for the Green function $ G$ in Section \ref{ssec:GF's estimates} and the results in this section also hold for the adjoint $ G^*$.
Let $\Omega$, $Y$ be as above and fix some $\xi \in Y \cap (\widetilde\Omega \setminus \Omega)$ with $R_{\xi} > 0$.

\begin{definition}
For $\eta \in (0,1)$ and any open set $U \subset X$, define the \emph{capacity width} $w_{\eta}(U)$ by
 \[ w_{\eta}(U) = \inf \left\{ r>0 : \forall x \in U, \frac{ {\mathrm{Cap}}_{B(x,2r)} \big( \overline{B(x,r)} \setminus U \big) }{ {\mathrm{Cap}}_{B(x,2r)} \big( \overline{B(x,r)} \big) } \geq \eta \right\} . \]
\end{definition}

Note that $w_{\eta}(U)$ is a decreasing function of $\eta \in (0,1)$ and an increasing function of the set $U$.

\begin{lemma} \label{lem:4.12} 
There are constants $A_7 \in (0,\infty)$ and $\eta \in (0,1/3]$ depending only on $D_Y$, $P_Y$, $c_u$, $C_u$,$C_0$-$C_7$, and upper bounds on $C_8 R^2$ and $A_{\frac{R}{2}}$, such that for all $0 < r < R \leq 2R_{\xi}$, 
 \[ w_{\eta} \big( \{ y \in B_{\widetilde\Omega}(\xi,R): d(y, \widetilde\Omega \setminus \Omega) < r \} \big)  \leq A_7r. \]
\end{lemma}

\begin{proof}
We follow \cite[Lemma 4.12]{GyryaSC}.
Let $Y_r = \{ y \in B_{\widetilde\Omega}(\xi,R): d(y, \widetilde\Omega \setminus \Omega) < r \}$ and $y \in Y_r$. Since $r < c_u \mathrm{diam}_{\Omega}(\Omega) / 12$, there exists a point $x \in \Omega$ such that $d_{\Omega}(x,y) = 4r / c_u$. By assumption, there is an inner uniform curve connecting $y$ to $x$ in $\Omega$. Let $z \in \Omega \cap \partial_{\Omega} B_{\Omega}(y,2r/c_u)$ be a point on this curve and note that $d_{\Omega}(y,z) = 2r/c_u \leq d_{\Omega}(x,y) - d_{\Omega}(y,z) \leq d_{\Omega}(x,z)$. Hence,
 \[ d_{\Omega}(z, \widetilde\Omega \setminus \Omega)  \geq  c_u \min\{d_{\Omega}(y,z),d_{\Omega}(z,x)\}  \geq  2r. \] 
So for any $y \in Y_r$ there exists a point $z \in \Omega \cap \partial_{\Omega} B_{\Omega}(y,2r/c_u)$ with $d(z,\widetilde\Omega \setminus \Omega) \geq 2r$. Thus, $B(z,r) \subset B(y,A_7 r) \setminus Y_r$ if $A_7 = 2/c_u + 1$. The capacity of $B(y,A_7 r) \setminus Y_r$ in $B(y,2A_7 r)$ is larger than the capacity of $B(z,r)$ in $B(y,2A_7 r)$, which is larger than the capacity of $B(z,r)$ in $B(z,3A_7 r)$. Thus, by Theorem \ref{thm:capacity estimate}, 
\begin{align*}  \frac{ {\mathrm{Cap}}_{B(y,2A_7r)} \big( \overline{B(y,A_7r)} \setminus Y_r \big) }{ {\mathrm{Cap}}_{B(y,2A_7r)} \big( \overline{B(y,A_7r)} \big) } 
& \geq \frac{ {\mathrm{Cap}}_{B(z,3A_7r)} \big( \overline{B(z,r)} \big) }{ {\mathrm{Cap}}_{B(y,2A_7r)} \big( \overline{B(y,A_7r)} \big) } \\
& \geq a \frac{A_7 r}{V(y,2A_7r)} \frac{1}{A} \frac{V(z,r)}{3A_7 r} 
\geq \frac{1}{3}.
\end{align*}
This shows that $w_{\eta}(Y_r) \leq \mathrm{const} \cdot A_7 r$ for some $\eta \in (0,1/3]$. Notice that the hypotheses of Theorem \ref{thm:capacity estimate} are satisfied because $B(z,6A_7r) \subset B(\xi,7A_7 R) \subset B(\xi,A_8 R_{\xi}) \subset Y$ by assumption.
\end{proof}

Write $w(U) := w_{\frac{1}{3}}(U)$ for the capacity width of an open set $U \subset \Omega$. 

The following lemma relates the capacity width to the $L$-harmonic measure $\omega$. A similar inequality holds for the $L^*$-harmonic measure $\omega^*$. We write $f \asymp g$ to indicate that $c g \leq f \leq C g$, for some constants $c,C \in (0,\infty)$ that depend only on $D_Y$, $P_Y$, $c_u$, $C_u$, $C_0$-$C_7$, and upper bounds on $C_8 R^2$ and $A_R$.

\begin{lemma} \label{lem:4.13}
There is a constant $a_1(D_Y, P_Y, C_0-C_7, C_8 R^2)$ such that for any non-empty open set $U \subset X$ and any ball $B(x,3r) \subset Y$ with $x \in U$, $0 < r < R$, we have
 \[ \omega_{U \cap B(x,r)}(x,U \cap (\overline{B(x,r)} \setminus B(x,r))) \leq \exp(2 - a_1 r / w(U)). \]
\end{lemma}

\begin{proof} We follow \cite[Lemma 1]{Aik01} and \cite[Lemma 4.13]{GyryaSC}. We may assume that $r/w(U) > 2$. For any $\kappa \in (0,1)$, we can pick $w(U) \leq s < w(U) + \kappa$ so that
 \[ \frac{ {\mathrm{Cap}}_{B(y,2s)}(\overline{B(y,s)} \setminus U) }{ {\mathrm{Cap}}_{B(y,2s)}(\overline{B(y,s)})} \geq \eta \quad \forall y \in U. \]
Consider a point $y \in U$ such that $B(y,3s) \subset Y$ and let $E = \overline{B(y,s)} \setminus U$. Let $\nu_E$ be the equilibrium measure of $E$ in $B = B(y,2s)$. We claim that there exists $A_2 > 0$ such that
\begin{equation} \label{eqh:biene}
  G_B \nu_E \geq A_2 \eta  \quad \textrm{ on } B(y,s).
\end{equation}
Let $F = \overline{B(y,s)}$ and $\nu_F$ be the equilibrium measure of $F$ in $B$. Then, by the Harnack inequality, for any $z$ with $d(y,z) = 3s/2$, we have
 \[  G_B(z,\zeta) \asymp  G_B(z,y)  \quad \forall \zeta \in B(y,s). \]
Hence,
 \[  G_B \nu_F(z) = \int_F  G_B(z,\zeta) \nu_F(d\zeta) \asymp  G_B(z,y) \nu_F(F) \]
and
 \[  G_B \nu_E(z) = \int_E  G_B(z,\zeta) \nu_E(d\zeta) \asymp  G_B(z,y) \nu_E(E). \]
Moreover, since $\nu_F(F) = {\mathrm{Cap}}_B(F)$, the two-sided inequality \eqref{eq:capacity estimate}
and Lemma \ref{lem:4.8} yield that $ G_B \nu_F(z) \simeq 1$. Hence, by definition of $s$, for any $z \in \overline{B(y,3s/2)} \setminus B(y,3s/2)$,
 \[  G_B \nu_E(z)  
\asymp  \frac{  G_B \nu_E(z) }{  G_B \nu_F(z) }
\asymp  \frac{ \nu_E(E) }{ \nu_F(F) } 
\asymp  \frac{ {\mathrm{Cap}}_B(E) }{ {\mathrm{Cap}}_B(F) }   \geq \eta. \]
This proves \eqref{eqh:biene}. 

Now, fix $x \in U$ such that $B(x,3r) \subset Y$. For simplicity, write 
 \[ \omega(\cdot) = \omega_{U \cap B(x,r)} (\cdot, U \cap (\overline{B(x,r)} \setminus B(x,r))). \]
Let $k$ be the integer such that $2kw(U) < r < 2(k+1)w(U)$, and pick $s > w(U)$ so close to $w(U)$ that $2ks < r$. We claim that
\begin{equation} \label{eqh:hummel}
 \sup_{U \cap B(x,r - 2js)} \{ \omega \}  \leq (1 - A_2 \eta)^j
\end{equation}
for $j = 0,1,\ldots,k$ with $A_2, \eta$ as in \eqref{eqh:biene}. Note that for $j=k$, \eqref{eqh:hummel} yields the inequality stated in this Lemma:
\[ \omega(x) \leq (1 - A_2 \eta)^k \leq \exp \left( \log( (1-A_2\eta)^{\frac{r}{2w(U)}} ) \right) 
 \leq e^2 \exp(- a_1 r/w(U)), \]
with $a_1 = - (\log( 1-A_2\eta)) / 2 $.

Inequality \eqref{eqh:hummel} is proved by induction, starting with the trivial case $j=0$. Assume that \eqref{eqh:hummel} holds for $j-1$. By the maximum principle, it suffices to prove
\begin{equation} \label{eqh:wespe}
 \sup_{U \cap (\overline{B(x,r - 2js)} \setminus B(x,r - 2js))} \{ \omega \}  \leq (1 - A_2 \eta)^j.
\end{equation}

Let $y \in U \cap (\overline{B(x,r - 2js)} \setminus B(x,r - 2js))$. Then $B(y,2s) \subset B(x,r-2(j-1)s)$ so that the induction hypothesis implies that 
 \[ \omega \leq (1-A_2\eta)^{j-1} \quad \textrm{ on } U \cap B(y,2s). \] 
Since $\omega$ vanishes (quasi-everywhere) on $(\overline{U} \setminus U) \cap B(x,r) \supset (\overline{U} \setminus U) \cap B(y,2s)$, the maximum principle implies that
\begin{align*}
\omega(b)
& =     \int_{\overline{(U \cap B(y,2s))} \setminus (U \cap B(y,2s))} \omega(a) \omega_{U \cap B(y,2s)}(b, da) \\
& \leq  (1-A_2\eta)^{j-1} \omega_{U \cap B(y,2s)}(b, U \cap (\overline{B(y,2s)} \setminus B(y,2s)))
\end{align*}
for any $b \in V \cap B(y,2s)$.
To estimate 
 \[ u = \omega_{U \cap B(y,2s)}(\cdot, U \cap (\overline{B(y,2s)} \setminus B(y,2s))), \]
on $U \cap B(y,2s)$, we compare it to
 \[ v = 1 -  G_{B(y,2s)} \nu_E, \]
where, as above, $\nu_E$ denotes the equilibrium measure of $E = \overline{B(y,2s)} \setminus U$ in $B(y,2s)$. Both functions are $ L$-harmonic in $U \cap B(y,2s)$, and it holds $u \leq v$ on $\overline{(U \cap B(y,2s))} \setminus (U \cap B(y,2s))$ quasi-everywhere (in the limit sense). By \eqref{eqh:biene}, this implies
 \[ u = \omega_{U \cap B(y,2s)}(\cdot, U \cap (\overline{B(y,2s)} \setminus B(y,2s)))  \leq  v  \leq  1 - A_2 \eta \]
on $U \cap B(y,2s)$. Hence, 
\[ \omega \leq (1 - A_2 \eta)^j  \quad \textrm{ on } U \cap B(y,2s). \]
Since this holds for any $y \in U \cap (\overline{B(x,r-2js)} \setminus B(x,r-2js))$, \eqref{eqh:wespe} is proved.
\end{proof}

\begin{lemma} \label{lem:4.14}
There exist constants $A_2, A_3 \in (0,\infty)$ depending only on $D_Y$, $P_Y$, $C_0$-$C_7$, $c_u$, $C_u$, and upper bounds on $C_8 R^2$ and $A_R$, such that for any $0 < r < R \leq R_{\xi}$
and any $x \in B_{\Omega}(\xi,r)$, we have
 \[  \omega \big( x, \Omega \cap \partial_{\Omega} B_{\widetilde\Omega}(\xi,2r), B_{\widetilde\Omega}(\xi,2r) \big) 
     \leq  A_2 \frac{ V(\xi,r) }{ r^2 }  G_{B_{\widetilde\Omega}(\xi,C_{\Omega} A_3 r)} (x,\xi_{16r}). \]
Here $\xi_{16r}$ is any point in $\Omega$ with $d_{\Omega}(\xi,\xi_{16r}) = 4r$ and 
\[  d(\xi_{16r}, X \setminus \Omega) = d(\xi_{16r}, X \setminus Y') \geq 2 c_u r. \]
A similar estimate holds for the $L^*$-harmonic measure $\omega^*$.
\end{lemma}

\begin{proof}
We follow \cite[Lemma 2]{Aik01} and \cite[Lemma 4.14]{GyryaSC}.
Let $A_3 = 2(12 + 12 C_u)$ so that all $(c_u,C_u)$-inner uniform paths connecting two points in $B_{\widetilde\Omega}(\xi,12r)$ stay in $B_{\widetilde\Omega}(\xi,A_3 r / 2)$. 
Recall that $Y' = B_{\widetilde\Omega}(\xi,A_0 r)$, where $A_0 = A_3+7$. 
For any $z \in B_{\widetilde\Omega}(\xi,A_3 r)$, set 
 \[  G'(z) =  G_{B_{\widetilde\Omega}(\xi,A_3 r)} (z,\xi_{16r}). \]
Let $s =  \min\{ c_u r, 5r/C_u \}$. 
Since
 \[  B_{\widetilde\Omega}(\xi_{16r},s)  \subset  B_{\widetilde\Omega}(\xi,A_3r) \setminus B_{\widetilde\Omega}(\xi,2r),  \]
the maximum principle yields
 \[  \forall y \in B_{\widetilde\Omega}(\xi,2r), \quad 
     G'(y)  \leq  \sup_{z \in \partial_{\Omega} B_{\widetilde\Omega}(\xi_{16r},s) }  G'(z). \]
Lemma \ref{lem:4.9} and the volume doubling condition yield
 \[  \sup_{z \in \partial_{\Omega} B_{\widetilde\Omega}(\xi_{16r},s) }  G'(z)   \leq  C \frac{r^2}{V(\xi,r)}, \]
for some constant $C>0$.
Hence, there exists $\epsilon_1 > 0$ such that
 \[ \forall y \in B_{\widetilde\Omega}(\xi,2r), \quad 
    \epsilon_1  \frac{V(\xi,r)}{r^2} G'(y)  \leq {e}^{-1}. \] 
Write
\begin{equation*}
 B_{\widetilde\Omega}(\xi,2r) = \bigcup_{j \geq 0} U_j \cap B_{\widetilde\Omega}(\xi,2r),
\end{equation*}
where
\[ U_j = \left\{ x \in Y' : \exp (-2^{j+1}) \leq \epsilon_1  \frac{V(\xi,r)}{r^2} G'(x) < \exp(-2^j) \right\}. \]
Let $ V_j = \left( \bigcup_{k \geq j} U_k \right) \cap B_{\widetilde\Omega}(\xi, 2r)$.
We claim that
\begin{equation} \label{eq:4.14} 
 w_{\eta}(V_j) \leq A_4 r \exp \left( - 2^j / \sigma \right)
\end{equation}
for some constants $A_4, \sigma \in (0,\infty)$.

Suppose $x \in V_j$. Observe that for $z \in \partial_{\Omega} B_{\widetilde\Omega}(\xi_{16r},s)$, by the inner uniformity of the domain, the length of the Harnack chain of balls in $B_{\widetilde\Omega}(\xi, A_3 r) \setminus \{ \xi_{16r} \}$ connecting $x$ to $z$ is at most $A_5 \log (1 + A_6 r / d(x, X \setminus Y'))$ 
for some constants $A_5, A_6 \in (0,\infty)$. Hence, there are constants $\epsilon_2, \epsilon_3, \sigma$ such that
\begin{align*}
 \exp (-2^j) 
& > \epsilon_1 \frac{V(\xi,r)}{r^2} G'(x)
\geq \epsilon_2 \frac{V(\xi,r)}{r^2} G'(z) \left( \frac{ d(x, X \setminus Y') }{ r } \right)^{\sigma} \\
& \geq \epsilon_3 \left( \frac{ d(x, X \setminus Y') }{ r } \right)^{\sigma}.
\end{align*}
The last inequality is obtained by applying \ref{lem:4.9} with $R = A_3r$ and $\delta = 5/A_3$. Now we have that for any $x \in V_j$,
 \[ d(x, X \setminus V_j) \leq d(x, X \setminus Y') \leq (\epsilon_3^{-1/\sigma} \exp(-2^j / \sigma) r) \wedge 2r. \]
This together with Lemma \ref{lem:4.12} yields \eqref{eq:4.14}.

Let $R_0 = 2r$ 
and for $j \geq 1$,
 \[ R_j = \left( 2 - \frac{6}{\pi^2} \sum_{k=1}^j \frac{1}{k^2} \right) r. \]
Then $R_j \downarrow r$ and
\begin{align} \label{eq:4.15}
 \sum_{j=1}^{\infty}  \exp \left( 2^{j+1} - \frac{ a_1 (R_{j-1} - R_j) }{ A_4 r \exp(-2^j / \sigma) } \right)
& = \sum_{j=1}^{\infty}  \exp \left( 2^{j+1} - \frac{ 6 a_1 }{ A_4 \pi^2 } j^{-2} \exp(2^j / \sigma) \right) \nonumber\\
& \leq \sum_{j=1}^{\infty}  \exp \left( 2^{j+1} - \frac{ 3 a_1 }{ C_{\Omega} A_4 \pi^2 } j^{-2} \exp(2^j / \sigma) \right) \nonumber \\
& < C < \infty.
\end{align}
Let $\omega_0 = \omega( \cdot, \Omega \cap \partial_{\Omega} B_{\widetilde\Omega}(\xi, 2r), B_{\widetilde\Omega}(\xi, 2r))$ and
 \[ d_j = \begin{cases} 
\sup \left\{ \frac{ r^2 \omega_0(x) }{ V(\xi,r) G'(x) }  : x \in U_j \cap B_{\widetilde\Omega}(\xi,R_j) \right\},  \quad 
    & \textrm{if } U_j \cap B_{\widetilde\Omega}(\xi, R_j) \neq \emptyset, \\
0,  & \textrm{if } U_j \cap B_{\widetilde\Omega}(\xi, R_j) = \emptyset.
\end{cases}
 \]
Since the sets $U_j \cap B_{\widetilde\Omega}(\xi,2r)$ cover $B_{\widetilde\Omega}(\xi,2r)$ and $B_{\widetilde\Omega}(\xi,r) \subset B_{\widetilde\Omega}(\xi,R_k)$ for each $k$, to prove Lemma \ref{lem:4.14}, it suffices to show that
 \[  \sup_{j \geq 0} d_j \leq A_2 < \infty \]
where $A_2$ is as in Lemma \ref{lem:4.14}.

We proceed by iteration. Since $\omega_0 \leq 1$, we have by definition of $U_0$,
 \[  d_0 = \sup_{ U_0 \cap B_{\widetilde\Omega}(\xi,2r) }  \frac{ r^2 \omega_0(x) }{ V(\xi,r) G'(x) }  \leq  \epsilon_1 {e}^2. \]
Let $j > 0$. For $x \in U_{j-1} \cap B_{\widetilde\Omega}(\xi, R_{j-1})$, we have by definition of $d_{j-1}$ that
 \[ \omega_0(x) \leq d_{j-1} \frac{  V(\xi,r) }{ r^2 } G'(x). \]
Also, $\omega_0 \leq 1$. Thus, the maximum principle yields that, for $x \in V_j \cap B_{\widetilde\Omega}(\xi,R_{j-1})$,
\begin{equation} \label{eq:4.16}
 \omega_0(x)  
\leq \omega( x, V_j \cap \partial_{\Omega} B_{\widetilde\Omega}(\xi,R_{j-1}), V_j \cap B_{\widetilde\Omega}(\xi,R_{j-1}) ) 
     + d_{j-1} \frac{ V(\xi,r) }{ r^2 } G'(x).
\end{equation}
For $x \in V_j \cap B_{\widetilde\Omega}(\xi,R_j)$, let $D = B(p(x),C_{\Omega}^{-1}(R_{j-1}-R_j))$ and let $D'$ be the connected component of $\widetilde \Omega$ that contains $x$ and so that $D' \cap \Omega$ is a connected component of $D \cap \Omega$. 
Let $\hat D = Y' \cap \Omega \cap D'$. Then by Lemma \ref{lem:metrics are comparable},
 \[ \hat D \subset D' \subset B_{\widetilde\Omega}(x,R_{j-1} - R_j) \subset B_{\widetilde\Omega}(\xi,R_{j-1}), \]
hence $\hat D \cap \partial_{\Omega} B_{\widetilde\Omega}(\xi,R_{j-1}) = \emptyset$. Thus, the first term on the right hand side of \eqref{eq:4.16} is not greater than
\begin{align*} 
& \omega \left(x, V_j \cap \hat D \cap \partial_X B\left(p(x),\frac{R_{j-1}-R_j}{2C_{\Omega}} \right), V_j \cap \hat D \cap B\left(p(x),\frac{R_{j-1}-R_j}{2C_{\Omega}}\right) \right) \\
& \leq \exp \left( 2 - \frac{a_1}{2C_{\Omega}} \frac{ R_{j-1} - R_j }{ w_{\eta}(V_j \cap \hat D) } \right) \\ 
& \leq \exp \left( 2 - \frac{a_1}{2C_{\Omega}} \frac{ R_{j-1} - R_j }{ w_{\eta}(V_j) } \right) \\
& \leq \exp \left( 2 - \frac{a_1}{2C_{\Omega} A_4} \exp( 2^j / \sigma) \frac{ R_{j-1} - R_j }{ r } \right) \\ 
& \leq \exp \left( 2 - \epsilon_6 j^{-2} \exp( 2^j / \sigma) \right)
\end{align*}
by Lemma \ref{lem:4.13}, monotonicity of $U \mapsto w_{\eta}(U)$ and \eqref{eq:4.14}. Here $\epsilon_6 = \frac{3 a_1}{\pi^2 A_4 C_{\Omega}}$, and $\partial_X B$ denotes the boundary of $B$ in $(X,d)$. 
Moreover, by definition of $U_j$,
 \[ \epsilon_1 \frac{V(\xi,r)}{r^2} G'(x) \geq \exp(-2^{j+1}) \]
for $x \in U_j$. Hence, for $x \in U_j \cap B_{\widetilde\Omega}(\xi,R_j)$, \eqref{eq:4.16} becomes
\begin{align*}
\omega_0(x)  & \leq  \exp \left( 2 - \epsilon_6 j^{-2} \exp( 2^j / \sigma) \right) 
                   + d_{j-1} \frac{ V(\xi,r) }{ r^2 } G'(x)  \\
             & \leq  \left( \epsilon_1 \exp \left( 2 + 2^{j+1} - \epsilon_6 j^{-2} \exp( 2^j / \sigma) \right) 
                   + d_{j-1} \right) \frac{ V(\xi,r) }{ r^2 } G'(x).
\end{align*}
Dividing both sides by $\frac{ V(\xi,r) }{ r^2 } G'(x)$ and taking the supremum over all points
$x \in U_j \cap B_{\widetilde\Omega}(\xi,R_j)$,
 \[ d_j  \leq   \epsilon_1 \exp \left( 2 + 2^{j+1} - \epsilon_6 j^{-2} \exp( 2^j / \sigma) \right) 
                   + d_{j-1}, \]
and hence for every integer $i > 0$,
 \[ d_i  \leq  \epsilon_1 {e}^2 \bigg( 1 + \sum_{j=1}^{\infty} \exp  \left(2^{j+1} - \frac{ 3 a_1 }{ \pi^2 A_4 C_{\Omega}} j^{-2} \exp( 2^j / \sigma) \right) \bigg) = \epsilon_1 {e}^2 (1 + C) < \infty \]
by \eqref{eq:4.15}. 
\end{proof}

\begin{proof}[Proof of Theorem \ref{thm:bHP with GF's}]
We follow \cite[Theorem 4.5]{GyryaSC} and \cite[Lemma 3]{Aik01}.
Recall that $A_0 = A_3 + 7 = 2(12 + 12 C_u) + 7$. 
Fix $\xi \in \widetilde\Omega \setminus \Omega$ with $R_{\xi} > 0$, let $0 < r < R \leq \inf \{ R_{\xi'} : \xi' \in B_{\widetilde\Omega}(\xi,7R_{\xi}) \setminus \Omega \}$ and set $Y' = B_{\widetilde\Omega}(\xi, A_0 r)$. 
Note that any two points in $B_{\widetilde\Omega}(\xi,12r)$ can be connected by a $(c_u, C_U)$-inner uniform path that stays in $B_{\widetilde\Omega}(\xi, A_3 r/2)$. 

Fix $x^* \in B_{\Omega}(\xi,r)$, $y^* \in \Omega \cap \partial_{\Omega} B_{\widetilde\Omega}(\xi,6r)$ such that $c_1 r \leq d(x^*,\widetilde\Omega \setminus \Omega) \leq r$ and $6c_0 r \leq d(y^*,\widetilde\Omega \setminus \Omega) \leq 6r$, for some constants $c_0,c_1 \in (0,1)$ depending on $c_u$ and $C_u$. Existence of $x^*$ and $y^*$ follows from the inner uniformity of $\Omega$. It suffices to show that for all $x \in B_{\Omega}(\xi,r)$ and $y \in \Omega \cap \partial_{\Omega} B_{\widetilde\Omega}(\xi,6r)$ we have
\begin{equation} \label{eq:4.17}
  G_{Y'}(x,y)  \asymp  \frac{  G_{Y'}(x^*,y) }{  G_{Y'}(x^*,y^*) }  G_{Y'}(x,y^*). 
\end{equation}

Fix $y \in \Omega \cap \partial_{\Omega} B_{\widetilde\Omega}(\xi,6r)$, and call $u$ ($v$, respectively) the left(right)-hand side of \eqref{eq:4.17}, viewed as a function of $x$. Then $u$ is positive and $L^*$-harmonic in $Y' \setminus \{ y \}$, whereas $v$ is positive and $L^*$-harmonic in $Y' \setminus \{ y^* \}$. Both functions vanish quasi-everywhere on the boundary of $Y'$. 

Since $y^* \in \Omega \cap \partial_{\Omega} B_{\widetilde\Omega}(\xi,6r)$ and $6c_0 r \leq d(y^*,\widetilde\Omega \setminus \Omega) \leq 6r$, it follows that the ball $B_{\widetilde\Omega}(y^*,3 c_0 r)$ is contained in $B_{\widetilde\Omega}(\xi,9r) \setminus B_{\widetilde\Omega}(\xi,3r)$. Let $z \in \partial_{\Omega} B_{\widetilde\Omega}(y^*, c_0 r)$. By a repeated use of Harnack inequality (a finite number of times, depending only on $c_u$ and $C_u$), one can compare the value of $v$ at $z$ and at $x^*$, so that by Lemma \ref{lem:4.9} (notice that $d(x^*,y) \geq c_1 r$) and the volume doubling property,
 \[ v(z) \leq C v(x^*) = C G_{Y'}(x^*,y) \leq C' \frac{r^2}{V(\xi,r)}. \]
Now, if $y \in B_{\widetilde\Omega}(y^*, 2c_0r)$, then by Lemma \ref{lem:4.9} (notice that $d_{\Omega}(z,y) \leq 3r \leq \frac{A_0 r}{6C_u}$ and $z,y \in B_{\widetilde\Omega}(\xi,A_0 r/6)$) and the volume doubling property,
 \[ u(z) =  G_{Y'}(z,y) \geq c \frac{r^2}{V(\xi,r)}, \]
so that we have $u(z) \geq c' v(z)$ in this case for some $c' > 0$. If instead $y \in \Omega \setminus B_{\widetilde\Omega}(y^*, 2c_0r)$, then we can connect $z$ and $x^*$ by a path of length comparable to $r$ that stays away (at scale $r$) from both $\widetilde\Omega \setminus \Omega$ and the point $y$. Hence, the Harnack inequality implies that $u(z) \asymp u(x^*) = v(x^*) \asymp v(z)$ in this case. This shows that we always have
 \[ u(z) \geq \epsilon_3 v(z)  \quad \forall z \in \partial_{\Omega} B_{\Omega}(y^*, c_0 r). \]
By the maximum principle, we obtain
 \[ u \geq \epsilon_3 v  \quad \textrm{ on } Y' \setminus B_{\widetilde\Omega}(y^*,c_0 r). \]
Since $B_{\Omega}(\xi,r) \subset Y' \setminus B_{\widetilde\Omega}(y^*,c_0 r)$, we have proved that $u \geq \epsilon_3 v$ on $B_{\widetilde\Omega}(\xi,r)$, that is,
\begin{equation}
  G_{Y'}(x,y) \geq \epsilon_3 \frac{  G_{Y'}(x^*,y) }{  G_{Y'}(x^*,y^*) }  G_{Y'}(x,y^*)
\end{equation}
for all $x \in B_{\Omega}(\xi,r)$ and $y \in \Omega \cap \partial_{\Omega} B_{\widetilde\Omega}(\xi,6r)$. This is one half of \eqref{eq:4.17}.

We now focus on the other half of \eqref{eq:4.17}, that is,
\begin{equation} \label{eq:4.19}
 \epsilon_4  G_{Y'}(x,y) \leq \frac{  G_{Y'}(x^*,y) }{  G_{Y'}(x^*,y^*) }  G_{Y'}(x,y^*),
\end{equation}
for all $x \in B_{\Omega}(\xi,r)$ and $y \in \Omega \cap \partial_{\Omega} B_{\widetilde\Omega}(\xi,6r)$.

For $x \in B_{\Omega}(\xi,2r)$ and $z \in B_{\widetilde\Omega}(\xi,9r) \setminus B_{\widetilde\Omega}(\xi,3r)$, Lemma \ref{lem:4.9} and the volume doubling condition yield
 \[  G_{Y'}(x,z) \leq C \frac{r^2}{V(\xi,r)}. \]
Regarding $ G_{Y'}(x,z)$ as $L$-harmonic function of $x$, the maximum principle gives
 \[  G_{Y'}(\cdot,z) \leq C \frac{ r^2 }{ V(\xi,r) } \omega( \cdot, \Omega \cap \partial_{\Omega} B_{\widetilde\Omega}(\xi,2r), B_{\widetilde\Omega}(\xi,2r))
\quad \textrm{ on }  B_{\widetilde\Omega}(\xi,2r). \]
Using Lemma \ref{lem:4.14} (note that $A_0 > A_3$) and the Harnack inequality (to move from $\xi_{16r}$ to $y^*$), we get for $x \in B_{\Omega}(\xi,r)$ and $z \in B_{\widetilde\Omega}(\xi,9r) \setminus B_{\widetilde\Omega}(\xi,3r)$, that
\begin{equation} \label{eq:4.20}
  G_{Y'}(x,z)  \leq  C A_2 \frac{r^2}{V(\xi,r)} \frac{V(\xi,r)}{r^2}  G_{Y'}(x,\xi_{16r})  \leq  C'  G_{Y'}(x,y^*),
\end{equation}
for some constant $C'\in(0,\infty)$.
Fix $x \in B_{\Omega}(\xi,r)$ and $y \in \Omega \cap \partial_{\Omega} B_{\widetilde\Omega}(\xi,6r)$. If $d_{\Omega}(y, X \setminus \Omega) \geq c_0 r/2$, then $ G_{Y'}(x,y) \asymp  G_{Y'}(x,y^*)$ and $ G_{Y'}(x^*,y) \asymp  G_{Y'}(x^*,y^*)$ by the Harnack inequality, so that \eqref{eq:4.19} follows. Hence we now assume that $y \in \Omega \cap \partial_{\Omega} B_{\widetilde\Omega}(\xi,6r)$ satisfies $d_{\Omega}(y, X \setminus \Omega) < c_0 r/2$. Let $\xi' \in \widetilde\Omega \setminus \Omega$ be a point such that $d_{\Omega}(y,\xi') < c_0 r/2$. It follows that $y \in B_{\Omega}(\xi',r)$. Also,
 \[ B_{\widetilde\Omega}(\xi',2r) \subset B_{\widetilde\Omega}(y,3r) \subset B_{\widetilde\Omega}(\xi,9r) \setminus B_{\widetilde\Omega}(\xi,3r). \]
We apply inequality \eqref{eq:4.20} to get $ G_{Y'}(x,z) \leq C_4  G_{Y'}(x,y^*)$ for any $z \in B_{\widetilde\Omega}(\xi',2r)$. Regarding $ G_{Y'}(x,y) =  G_{Y'}^*(y,x)$ as $ L^*$-harmonic function of $y$, we obtain
\begin{equation} \label{eq:4.21}
  G_{Y'}(x,y) \leq C_4  G_{Y'}(x,y^*) \, \omega^*( y, \Omega \cap \partial_{\Omega} B_{\widetilde\Omega}(\xi',2r), B_{\widetilde\Omega}(\xi',2r) ).
\end{equation}
Let us apply Lemma \ref{lem:4.14} with $\xi$ replaced by $\xi'$. This yields
\begin{align}
 \omega^* (y, \Omega \cap \partial_{\Omega} B_{\widetilde\Omega}(\xi',2r), B_{\widetilde\Omega}(\xi',2r) )
& \leq  A_2 \frac{V(\xi',r)}{r^2}  G^*_{B_{\widetilde\Omega}(\xi',C_{\Omega} A_3r)}(y,\xi'_{16r}) \nonumber \\ 
& \leq  A_2' \frac{V(\xi,r)}{r^2}  G_{Y'}(\xi'_{16r},y),  \label{eq:4.22}
\end{align}
where $\xi'_{16r} \in \Omega$ is any point such that $d_{\Omega}(\xi'_{16r},\xi') = 4r$ and $d(\xi'_{16r}, X \setminus \Omega) \geq 2 c_u r$. Observe that we have used the volume doubling property as well as the set monotonicity of the Green function, and that $B_{\widetilde\Omega}(\xi',A_3r) \subset B_{\widetilde\Omega}(\xi,A_0r)$ because $A_0 = A_3 + 7$ and $d_{\Omega}(\xi,\xi') \leq 7r$. Now, \eqref{eq:4.21} and \eqref{eq:4.22} give
\begin{equation} \label{eq:zettel}
  G_{Y'}(x,y) \leq C_5 \frac{V(\xi,r)}{r^2}  G_{Y'}(\xi'_{16r},y)  G_{Y'}(x,y^*).
\end{equation}
By construction, $d_{\Omega}(\xi'_{16r},y) \geq d(\xi'_{16r},\xi') - d_{\Omega}(\xi',y) \geq 2r$ and $d_{\Omega}(x^*,y) \geq d_{\Omega}(\xi,y) - d_{\Omega}(\xi,x^*) \geq 5r$. Using the inner uniformity of $\Omega$, we find a chain of balls, each of radius $\asymp r$ and contained in $Y' \setminus \{ y \}$, going from $x^*$ to $\xi'_{16r}$, so that the length of the chain is uniformly bounded in terms of $c_u, C_u$. Applying the Harnack inequality repeatedly thus yields $ G_{Y'}(\xi'_{16r},y) \asymp  G_{Y'}(x^*,y)$. As Lemma \ref{lem:4.9} gives $ G_{Y'}(x^*,y^*) \asymp r^2 / V(\xi,r)$, inequality \eqref{eq:zettel} implies \eqref{eq:4.19}. This completes the proof. 
\end{proof}

\bibliographystyle{siam}
\bibliography{bibliography}
\end{document}